\documentclass[10pt]{article}

\usepackage{amsmath}
\usepackage{amssymb}
\usepackage{amsthm}
\usepackage[english]{babel}
\usepackage{dsfont}
\usepackage{float}
\usepackage{hyperref}
\usepackage{subcaption}
\usepackage{tikz}
\usepackage{url}

\usepackage{cleveref}
\usepackage[utf8]{inputenc}
\usepackage{pgfplots}

\usepackage[a4paper, total={30.5pc, 51pc}]{geometry}
\hypersetup{
	colorlinks=true,
	linkcolor=red,
	filecolor=magenta,      
	urlcolor=cyan,
	citecolor=green,
	linktoc=all,
}
\pgfplotsset{compat=1.18}

\crefformat{equation}{(#2#1#3)}
\newcommand{\argdot}{\,\cdot\,}

\newcommand{\E}{\mathbb{E}}
\newcommand{\F}{\mathbb{F}}
\newcommand{\N}{\mathbb{N}}
\renewcommand{\P}{\mathbb{P}}
\newcommand{\Q}{\mathbb{Q}}
\newcommand{\R}{\mathbb{R}}
\renewcommand{\S}{\mathbb{S}}
\newcommand{\T}{\mathbb{T}}
\newcommand{\V}{\mathbb{V}}

\newcommand{\cA}{\mathcal{A}}
\newcommand{\cB}{\mathcal{B}}
\newcommand{\cC}{\mathcal{C}}
\newcommand{\cE}{\mathcal{E}}
\newcommand{\cF}{\mathcal{F}}

\newcommand{\cJ}{\mathcal{J}}
\newcommand{\cL}{\mathcal{L}}
\newcommand{\cN}{\mathcal{N}}
\newcommand{\cS}{\mathcal{S}}
\newcommand{\cU}{\mathcal{U}}
\newcommand{\cV}{\mathcal{V}}

\newcommand{\cY}{\mathcal{Y}}

\newcommand{\mfF}{\mathfrak{F}}
\newcommand{\mfM}{\mathfrak{M}}
\newcommand{\mfS}{\mathfrak{S}}

\newcommand{\rd}{\mathrm{d}}
\newcommand{\rD}{\mathrm{D}}
\newcommand{\rF}{\mathrm{F}}

\newcommand{\intS}{\mathring{\S}}

\theoremstyle{plain}
\newtheorem{theorem}{Theorem}[section]
\newtheorem{proposition}[theorem]{Proposition}

\newtheorem{lemma}[theorem]{Lemma}
\newtheorem{definition}[theorem]{Definition}
\newtheorem{assumption}{Assumption}

\theoremstyle{definition}
\newtheorem{remark}[theorem]{Remark}

\DeclareMathOperator{\trace}{tr}
\DeclareMathOperator{\diag}{diag}

\title{Optimal adaptive control with separable drift uncertainty}
\author{
Samuel N.\ Cohen%
\thanks{University of Oxford, Mathematical Institute, Oxford, United Kingdom,
		OX2 6GG (\url{cohens@maths.ox.ac.uk}).}
\and
Christoph Knochenhauer%
\thanks{Technische Universit\"{a}t München, School of Computation,
        Information and Technology,
		Parkring 11--13, 85748 Garching bei M\"{u}nchen, Germany
		(\url{knochenhauer@tum.de}).}
\and
Alexander Merkel%
\thanks{Technische Universit\"{a}t Berlin, Institut f\"{u}r Mathematik,
		Stra\ss{}e des 17. Juni 136, 10623 Berlin, Germany
		(\url{merkel@math.tu-berlin.de}).}
}

\begin{document}

\maketitle

\begin{abstract}
We consider a problem of stochastic optimal control with separable drift
uncertainty in strong formulation on a finite time horizon. The drift of
the state $Y^{u}$ is multiplicatively influenced by an unknown random variable
$\lambda$, while admissible controls $u$ are required to be adapted to the
observation filtration. Choosing a control actively influences the state and
information acquisition simultaneously and comes with a learning effect. The
problem, initially non-Markovian, is embedded into a higher-dimensional
Markovian, full information control problem with control-dependent filtration
and noise. To that problem, we apply the stochastic Perron method to
characterize the value function as the unique viscosity solution of the HJB
equation, explicitly construct $\varepsilon$-optimal controls, and show that the
values in the strong and weak formulation agree. Numerical illustrations show a
significant difference between the adaptive control and the certainty
equivalence control, highlighting a substantial learning effect.
\end{abstract}

\section{Introduction}

Active learning in stochastic control is a topic of considerable interest,
particularly in situations where unobservable components can affect the state
evolution. In Feldbaum's seminal work \cite{feldbaum1960dual}, the concept of
the {\it dual effect} was introduced (see \cite{bar1974dual} for a treatment in
the field of stochastic control). The dual effect is the interplay between the
control's effect on the state and its influence on the estimation of
unobservable components through the controlled state. Consequently, the dual
effect plays a key role in problems with a learning effect in stochastic
control, as it describes a case of what is called, in ``modern'' language, the
much-studied trade-off between exploration and exploitation. Here, exploration
is in terms of knowledge / uncertainty about the unobservable component and
exploitation refers to cost optimization.

In this paper, we investigate the problem of Bayesian adaptive optimal
stochastic control in continuous time on a finite time horizon with separable
drift uncertainty introduced via a hidden, static random variable. We take a
Bayesian view of the estimation problem, that is we assume the prior is known
and subsequently update our beliefs. In this context, $\varepsilon$-optimal
controls are constructed using stability of viscosity solutions, and strong and
weak formulations are shown to agree in value.

\subsection{Problem description}

In the following, the controlled state $Y^u$ satisfies
\[
  \rd Y_s^u = \lambda^\intercal b(s,Y_s^u,u_s)\rd s
  				+ \sigma(s,Y_s^u,u_s)\rd W_s,\qquad Y_0=0
\]
where $u$ is the control and $\lambda$ is an unobservable random variable with
prior distribution $\mu$. For each control, the state generates a filtration
$\cY^u = \sigma(Y^u)$ called the ``observation filtration'' which is explicitly
control-dependent in the strong formulation. As a result of the nonlinear drift,
the {\it separation principle} first formulated in \cite{wonham1968separation}
generally does not hold. The separation principle roughly says that, under
certain conditions (typically linearity of the state dynamics), control and
estimation can be decoupled, and the problem of simultaneous control and
estimation separates into two problems; this does not apply here. The goal is to
minimize the cost functional
\[
  \cJ(u) = \E\Bigl[\int_0^T k(t,Y_t^u,u_t,\lambda)\rd t
  			+ g(Y_T^u,\lambda)\Bigr]
\]
over a class of controls $u$ which are adapted to their own generated
filtration $\cY^u$, that is, they only rely on the information on $\lambda$
and $W$ generated from observing $Y^u$.

In this formulation, the control is of closed-loop type and directly influences
the controller's knowledge of the hidden parameter $\lambda$, resulting in a
learning effect. More precisely, a trade-off between exploration and
exploitation arises, as the control must balance improving the estimation and
minimizing the cost functional.

Mathematically, the dependence of the observation filtration $\cY^u$ on the
control $u$ is a result of the strong formulation of the control problem. To
make the problem approachable using techniques of stochastic optimal control,
especially dynamic programming, we rewrite the state dynamics in the filtration
$\cY^u$. It turns out that, by introducing two control-dependent auxiliary
states, we can fully describe the conditional distribution of $\lambda$ given
$\cY^u$ in a Markovian way. As such, the original problem is embedded into a
finite-dimensional Markovian problem, and we can apply the techniques of dynamic
programming.

To gain further insight, we also study an alternative weak formulation of the
problem. A priori, given the effects of controlling the information flow and the
dependence of the filtration on the control in the strong formulation, it is not
clear if the optimal cost in the weak and strong formulation agree.
Nevertheless, we link the two formulations using the stochastic Perron method
\cite{bayraktar2013stochastic}, which yields a characterization of the
strong and weak value functions as the unique viscosity solution of the same HJB
equation, that is, the two value functions agree. The stochastic Perron method
allows for the derivation of the viscosity characterization of the value
functions without explicitly proving the dynamic programming principle (DPP) in
continuous time. Instead, one uses suitable notions of sub- and supersolutions
carrying the necessary intertemporal structure required for the proof.

Choosing a family of auxiliary control problems with the control set restricted
to piecewise constant controls as the class of supersolutions, we establish the DPP
for these controls and then build an approximation scheme in the sense of
\cite{barles1991convergence} converging from above to a viscosity subsolution
of the HJB equation which dominates the value function in the strong
formulation. Regarding the approximation from below, we consider stochastic
subsolutions (in the weak formulation) and show that their pointwise supremum is
a viscosity supersolution dominated by the value function in the weak
formulation. Consequently, a comparison principle for the HJB
equation implies that the limit of the control problems with piecewise constant
controls and \emph{the value functions of the problems in strong and weak
formulation agree}, which is to say that additional randomization does not
decrease the value of the control problem. Moreover, by establishing existence
of optimizers for the problems with piecewise constant controls, we are able to
\emph{construct $\varepsilon$-optimal controls} which can be efficiently
computed.

Problems with unknown dynamics and cost have highly relevant applications in,
for example, the problem of optimal execution in mathematical finance. In this
context, $Y^u$ represents the asset price under price impact and $b(t,y,u)=u$
is a simple model for unobservable permanent price impact $\lambda$
(see the monographs
\cite{cartea2015algorithmic,gueant2016financial,webster2023handbook} for an
introduction to problems of this type). This is a challenging issue in problems
of optimal execution, as price impact factors are generally unobservable and
have to be estimated from the affected price. Another field of application is in
motion control of robots; we solve a stylized example numerically in
\Cref{sec:Appl} to demonstrate and compare our results to a na\"{i}ve
control obtained by replacing $\lambda$ by its expectation $\E[\lambda]$ and a
certainty equivalent (CE) control which uses the current conditional mean as
the best estimate, but neglects the control of future available information.
The HJB equation is solved numerically using the deep Galerkin method
\cite{sirignano2018dgm} combined with policy iteration; convergence of this
method in the linear case was established recently in \cite{jiang2023global}.

\subsection{Related literature}

Previous studies involving unobservable components, for example in other
problems of adaptive control or partially observable control (e.g.\
\cite{benevs1991separation,bensoussan1992stochastic,fleming1982optimal}),
usually pose the problem in a weak formulation such that the filtration does not
depend on the control. There, the probability space and filtration are fixed,
weak solutions to the state equation are considered, and the control is
introduced via a change of measure to the cost functional. In the stochastic
optimal control literature, direct dependence of the observation filtration on
the control is usually avoided or resolved via the separation principle.
We are not aware of any work studying such problems in the strong
formulation when the separation principle does not hold.

In \cite{benevs1991separation,karatzas1992resolvent} on an infinite horizon and
in \cite{karatzas1993finite} on a finite horizon, a special case of our control
problem is considered in the weak formulation for quadratic cost on the state
variable in a class of what the authors refer to as ``wide-sense" admissible
controls. They show the intriguing result that there
does not exist an optimal control in the class of ``strict-sense" admissible
controls in the weak formulation but construct optimal controls in the class of
``wide-sense" admissible controls. This, together with the fact that the optimal
wide-sense admissible control is Markovian, suggests also that, in the strong
formulation, in general there does not exist an optimal admissible control, thus
justifying our search for $\varepsilon$-optimal controls.

A survey on the stochastic adaptive control problem from a control theory
perspective is given in \cite{kumar1985survey} and a recent broader view on
adaptive control is, e.g., the monograph \cite{astolfi2008nonlinear}. There, the
goal is generally not to find optimal controls but implementable ``good
controls" having ``good" asymptotic stability and robustness properties and to
prove suboptimality guarantees. They also treat non-Bayesian approaches. Often,
for example, an ergodic criterion is prescribed, and we mention
\cite{duncan1999adaptive,duncan1990adaptive} as examples in this direction.

The problem of adaptive control has also attracted attention over the last
decade due to the interest in reinforcement learning and, more generally,
machine learning and its applications in control theory. We mention two recent
examples: \cite{szpruch2021exploration} who examine the linear convex episodic
reinforcement learning problem in continuous time and \cite{mania2019certainty}
who derive suboptimality bounds for a certainty equivalent controller in the
linear-quadratic problem in discrete time. The reinforcement learning community
mostly considers asymptotic regret optimality and other types of asymptotic
optimality, whereas we consider optimality on a finite horizon. As a
consequence, in our formulation control effort matters at every point in time,
whereas for asymptotic optimality, control effort on every finite time interval
is irrelevant.

Finally, regarding our construction of $\varepsilon$-optimal controls, we point
out that a way to construct $\varepsilon$-optimal controls in the class of
piecewise constant controls was already suggested in the seminal monograph
\cite{krylov2008controlled}. In \cite{krylov1999approximating} the author
approximates the value function using value functions over the class of
piecewise constant controls and proves the DPP in that class. We use regularity
results appearing in \cite{krylov2008controlled} to establish the DPP
for the case of unbounded cost functions over the class of
piecewise constant controls. As a consequence of measurable selection, we obtain
optimal controls for the approximating problems, which are Markovian ``on a
time grid". Convergence, and thus $\varepsilon$-optimality, for the original
problem is not shown via direct analysis of the cost functional as in
\cite{krylov2008controlled} and \cite{krylov1999approximating}, but using
instead the theory of viscosity solutions and the main stability result of
\cite{barles1991convergence}, allowing us to additionally characterize the
value function as the unique continuous viscosity solution of the HJB
equation. The connection to the HJB equation was not made in
\cite{krylov1999approximating}, but explicit convergence rates in terms of the
step size were obtained, something that is not included in our approach.

The rest of our work is structured as follows. In \Cref{sec:theOP} we
formulate the optimal control problem in strong formulation. In
\Cref{sec:trans-FI} we rewrite the state's dynamics in its own filtration
and introduce two additional auxiliary states. A dynamic Markovian optimal
control problem in strong formulation is introduced into which the original
problem is embedded. In \Cref{sec:Weak-OP} we furthermore introduce the
Markovian control problem in weak formulation. In \Cref{sec:DPP-eps-opt}
we establish the main results of this paper by showing that the weak and strong
value functions coincide with the unique continuous viscosity solution of the
HJB equation and construct $\varepsilon$-optimal
controls. Finally, in \Cref{sec:Appl} we present an application of our results
to a toy problem of optimal control in robotics, highlighting a substantial
learning effect. \Cref{sec:Appendix} gathers several proofs which we consider
classical but chose to include to keep the paper self-contained.

\subsection*{Notation}

Throughout, we fix a complete probability space $(\Omega,\mfF,\P)$ and denote by
$T>0$ a finite time horizon. The symbols $\rD_x,\rD_{xy}$ denote the gradient
and Hessian with respect to the (multivariate) components~$x,y$, whereas
$\partial_z$ denotes the partial derivative with respect to the (scalar)
component~$z$. Finally, $\cS_d$ denotes the set of symmetric $d\times d$
matrices for any $d\in\N$.

\section{The control problem}\label{sec:theOP}

We pose the control problem beginning with an $\R^m$-valued random variable
$\lambda=(\lambda_1,\dots,\lambda_m)^\intercal$ with distribution $\mu$ under
$\P$. We furthermore suppose that $(\Omega,\mfF,\P)$ supports a one-dimensional
Brownian motion $W=(W_{t})_{t\in[0,T]}$ independent of $\lambda$, and we denote
by $\cF^{\lambda,W}$ the filtration generated by $\lambda$ and $W$, augmented
by all $\P$-nullsets.

\begin{assumption}\label{asspt-mu}
$\lambda$ is bounded, that is $|\lambda|\leq K$ for some $K\geq0$.
\end{assumption}

This assumption guarantees that the estimator function introduced in the next
section is Lipschitz continuous. Next, we consider controls taking values in
a compact metric space $\cU$. This is a standard
assumption which allows us to construct optimal controls using measurable
selection.
With this, the set of \emph{pre-admissible controls} is
\[
  \cA^{pre}:=\bigl\{u:\Omega\times[0,T]\to \cU: u \text{ is }
  \cF^{\lambda,W}\text{-progressively measurable}\bigr\}.
\]
For any $u\in\cA^{pre}$, the controller observes a controlled one-dimensional
state process $Y^u=(Y^u_t)_{t\in[0,T]}$, defined as the unique strong solution
of
\begin{equation}\label{eq:state-proc}
  \rd Y_t^u=\lambda^\intercal b(t,Y_t^u,u_t)\rd t
    + \sigma(t,Y_t^u,u_t)\rd W_t,\qquad Y_0=0,
\end{equation}
where $b:[0,T]\times\R\times\cU\to\R^{m}$ and
$\sigma:[0,T]\times\R\times\cU\to(0,\infty)$, and the initial state $Y_0=0$ is
chosen for simplicity.

\begin{assumption}\label{asspt-coeff}
The functions $b,\sigma$ are jointly continuous in all arguments. Furthermore,
there exist constants $L,M>0$ such that, for all $t\in[0,T], u\in\cU$,
\begin{align*}
  |b(t,y_1,u)-b(t,y_2,u)|+|\sigma(t,y_1,u)-\sigma(t,y_2,u)| 
    &\leq L|y_1-y_2|, & \forall y_1,y_2&\in\R,\\
  |b(t,y,u)|+|\sigma(t,y,u)|+|\sigma(t,y,u)^{-1}| &\leq M, &\forall y&\in\R.
\end{align*}
\end{assumption}

Lipschitz-continuity and boundedness ensure existence of a strong solution of
\cref{eq:state-proc} and boundedness of $b,\sigma^{-1}$ ensure that
a certain Girsanov transform used to derive an estimator for $\lambda$ is
valid; boundedness of $\sigma$ is for convenience. Uniformity of these
estimates is required for regularity of the approximating problems in
\Cref{subsec:infimum-supersolutions}.

By a standard existence result such as
\cite[Theorem 2.5.7]{krylov2008controlled}, for each pre-admissible control
$u\in\cA^{pre}$, there exists a pathwise unique $(\cF^{\lambda,W},\P)$-strong
solution of \eqref{eq:state-proc} which generates a filtration, called the
\emph{observation filtration}, which we highlight to be control-dependent in the
strong formulation.

\begin{definition}
For $u\in\cA^{pre}$, the observation filtration $\cY^u=(\cY_t^u)_{t\in[0,T]}$
is defined as the completed filtration generated by $Y^u$, that is 
$\cY_t^u:=\sigma(Y_s^u,s\in[0,t])\vee\cN$ for all $t\in[0,T]$, where $\cN$
denotes the system of $\P$-nullsets.
\end{definition}

\begin{remark}
\begin{itemize}
\item[1)] By definition, for every $u\in\cA^{pre}$, we have
$\cY^u\subset\cF^{\lambda,W}$.
\item[2)] In the state dynamics \cref{eq:state-proc}, the unobservable
$\lambda$ does not appear in the diffusion coefficient. If it did, the problem
would be fundamentally different.
\end{itemize}
\end{remark}

We now restrict the controls to those which are adapted to their
corresponding filtration $\cY^u$, that is, they only rely on the
information on $\lambda$ and $W$ obtained from observing the controlled state.

\begin{definition}
The set of admissible controls is defined as
\[
  \cA:=\bigl\{u\in\cA^{pre} : u\text{ is }\cY^u\text{-progressively
    measurable}\bigr\}. 
\]
\end{definition}

Admissible controls are therefore {\it closed-loop controls} in the sense that
they can utilize their effect on the observations; see \cite{bar1974dual} for
an elaborate discussion.

\begin{remark}
Care must be taken in defining the set of admissible controls, as can be seen
from the example of the Tsirel'son SDE (see, e.g.,
\cite[p.362]{revuz2013continuous}). There, a bounded, nonanticipating but
path-dependent drift (potentially a feedback map) is constructed, which
introduces additional independent randomness in the generated filtration, such
that ``$\cY^{u}\not\subseteq \cF^{\lambda,W}$". Our choice of admissible
controls, specifically $\cA\subset\cA^{pre}$, excludes these controls from
being admissible.

In general, a definition of the set of admissible controls adapted to the
filtration of a solution of a controlled SDE is circular. Indeed, in order for
the observation filtration to exist, for each control there needs to exist a
solution of the state equation \cref{eq:state-proc}, but to ensure such an
existence, one needs to specify the set of admissible controls. To the best of
our knowledge, there are three approaches:
\begin{itemize}
\item Fix an observation filtration $\cF$ and work with weak solutions to the
state equation and controls $u$ which are $\cF$-progressive (as in, e.g.,
\cite{benevs1991separation,fleming1982optimal}).
\item Choose as controls nonanticipative feedback maps
$\hat{u}:[0,T]\times\cC([0,T])\to\cU$, regular enough to define the state
process. This is an approach often followed by the reinforcement
learning community.
\item Use a ``reference" filtration, in our case $\cF^{\lambda,W}$, to define a
superset of controls, in our case $\cA^{pre}$, and guarantee existence of the
controlled state. Then restrict to those controls that are adapted to a
control-dependent subfiltration, in our case $\cY^u$.
\end{itemize}

Note that the set of controls $\cA$ is not a nice set to work with. For example,
it is not closed under addition, even when the sum takes values in $\cU$;
i.e.\ for $u_1,u_2\in\cA$ it is not clear whether $u_1+u_2$ is
$\cY^{u_1+u_2}$-progressive and thus admissible.
\end{remark}

For a control $u\in\cA$, we define the \emph{cost functional} as
\begin{equation}\label{eq:cost-functional-primal}
  \cJ(u):=\E\Bigl[\int_0^T k(t,Y_t^u,u_t,\lambda)\rd t + g(Y_T^u,\lambda)\Bigr]
    \qquad
	\text{subject to }\cref{eq:state-proc},
\end{equation}
where $k:[0,T]\times\R\times\cU\times\R\to[0,\infty)$ and
$g:\R\times\R\to[0,\infty)$. The goal is to minimize the cost functional
$\cJ(u)$ over all $u\in\cA$.

\begin{assumption}\label{asspt-cost}
$k$ and $g$ are jointly continuous. In addition, $k(t,y,u,\ell)$ is continuous
in $y$ uniformly over $u$ for each $(t,\ell)\in[0,T]\times\R$ and $g(y,\ell)$ is
uniformly continuous in $y$ for each $\ell$. Furthermore, we assume that there
exist $C,p>0$ with
\[
  |k(t,y,u,\ell)|+|g(y,\ell)|\leq C(1+|y|^p)
    \qquad\forall(t,y,u,\ell)\in[0,T]\times\R\times \cU\times\R.
\]
\end{assumption}

\begin{definition}\label{def:optimality}
We say that a control $u^*\in\cA$ is optimal if
\begin{equation}\label{eq:original-OP}
  \inf_{u\in\cA}\cJ(u)=\cJ(u^*)
\end{equation}
with $\cJ$ as in \cref{eq:cost-functional-primal}. Moreover, for every
$\varepsilon>0$, a control $u^{*,\varepsilon}\in\cA$ is $\varepsilon$-optimal
provided that $\cJ(u^{*,\varepsilon})\leq\inf_{u\in\cA}\cJ(u)+\varepsilon$.
\end{definition}

\begin{remark}
In the problem formulation considered here, $Y^u$ is generally not a
$\cY^u$-Markov process. In the next section, we derive a finite-dimensional
control problem under full information where for each control $u\in\cA$ the
coefficients of $Y^u$ are adapted to $\cY^u$, and which can be solved via
dynamic programming in the sense of $\varepsilon$-optimal controls. We thus
also obtain $\varepsilon$-optimal controls for the original 
problem~\cref{eq:original-OP}.
\end{remark}

\begin{remark}
\begin{itemize}
\item[1)] Our results extend to the case of a multidimensional state $Y^u$
under suitably adjusted assumptions. Since this only adds to the notation, we
stick to the one-dimensional case.
\item[2)] A motivating example for a multidimensional hidden parameter $\lambda$
can be constructed by considering
$\lambda=(1,\bar{\lambda},\bar{\lambda}^2,\dots,\bar{\lambda}^N)$ for some
real-valued $\bar{\lambda}$ and fixed $N\in\N$. This, together
with suitable coefficients $b_i(\argdot)$, $i=0,\dots,N$, could be used for a
polynomial approximation of a general drift function
$b(\argdot,\bar{\lambda})\approx\lambda^\intercal b(\argdot)$. The validity of
such an approximation is left open for future research, but may provide a way
to study the case of non-separable drift uncertainty.
\end{itemize}
\end{remark}

\section{Transformation to a full information problem}\label{sec:trans-FI}

We now embed the problem into one in which the coefficients and cost are
adapted to the observation filtration. Making use of techniques of Bayesian
inference, we find a finite-dimensional parametrization of the conditional
distribution of $\lambda$ by two $\cY^u$-adapted information states. This
suffices to transform the problem into a control problem under full information.

To begin with, we introduce a new probability measure $\Q^u$ as follows. Let
$u\in\cA$ and define two density processes
$\Lambda^u=(\Lambda^u_t)_{t\in[0,T]}$ and $Z^u:=(Z^u_t)_{t\in[0,T]}$ via
\begin{align*}
  \Lambda^u_t:=\frac{1}{Z_t^u}
  &:=\cE\Bigl(-\int_0^{\argdot}\frac{\lambda^\intercal b}{\sigma}(s,Y_s^u,u_s)
    \rd W_s\Bigr)_t\\
  &=\exp\Bigl(-\int_0^t \frac{\lambda^\intercal b}{\sigma}(s,Y_s^u,u_s)
    \rd W_s-\frac{1}{2}\int_0^t
    \frac{\lambda^\intercal bb^\intercal\lambda}{\sigma^2}(s,Y_s^u,u_s)
    \rd s\Bigr)\\
  &=\exp\Bigr(-\int_0^t \frac{\lambda^\intercal b}{\sigma^2}(s,Y_s^u,u_s)
    \rd Y_s^u+\frac{1}{2}\int_0^t
    \frac{\lambda^\intercal bb^\intercal\lambda}{\sigma^2}(s,Y_s^u,u_s)
    \rd s\Bigr).
\end{align*}
Since $\lambda$ is bounded by \Cref{asspt-mu} and $b,\sigma^{-1}$ are bounded
by \Cref{asspt-coeff}, $\Lambda^u$ is an $(\cF^{\lambda,W},\P)$-martingale and
defines a $\P$-equivalent probability measure $\Q^u$ on
$(\Omega,\cF^{\lambda,W}_T)$ with
\[
  \Q^u(A):=\E[\Lambda^u_T\mathds{1}_A]\qquad\forall A\in\cF^{\lambda,W}_T.
\]
By Girsanov's theorem \cite[Theorem 3.5.1]{karatzas1998brownian}, we find
that $Y^u$ is a standard $(\cF^{\lambda,W},\Q^u)$-Brownian motion independent
of $\lambda$ and, as $\cY^u\subset\F^{\lambda,W}$, it is also a
$(\cY^u,\Q^u)$-Brownian motion. Furthermore, $\lambda$ retains distribution
$\mu$ under $\Q^u$; see \cite[Lemma 2.2]{karatzas2001bayesian}. With this,
we define another density process $\hat{Z}^u=(\hat{Z}^u_t)_{t\in[0,T]}$, which
is also a $(\cY^u,\Q^u)$-martingale, by
\begin{align}
  \hat{Z}^u_t &:= \int_\R\exp\Bigl(\int_0^t \frac{\ell^\intercal b}{\sigma^2}
    (s,Y_s^u,u_s)\rd Y_s^u
    -\frac{1}{2}\int_0^t \frac{\ell^\intercal bb^\intercal\ell}{\sigma^2}
    (s,Y_s^u,u_s)\rd s\Bigr)\mu(\rd\ell)\label{eq:def-hat-Z}\\
  &=\E^{\Q^u}\Bigl[\exp\Bigl(\int_0^t\frac{\lambda^\intercal b}{\sigma^2}
    (s,Y_s^u,u_s)\rd Y_s^u
    -\frac{1}{2}\int_0^t\frac{\lambda^\intercal bb^\intercal\lambda}{\sigma^2}
    (s,Y_s^u,u_s)\rd s\Bigr)\Big|\cY^u_t\Bigr]\notag\\
  &=\E^{\Q^u}[Z^u_T|\cY^u_t]=\E^{\Q^u}[Z^u_t|\cY^u_t],\notag
\end{align}
where we used that $\lambda\sim\mu$ and is independent of $Y^u$ under $\Q^u$.

In order to avoid having to deal with matrix-valued SDEs, we introduce the
vectorization operator $\Xi:\cS_m\to\R^{m(m+1)/2}$ represented by the
$\{0,1\}$-valued three-tensor $\Xi\in\R^{m\times m\times m(m+1)/2}$ such that
\[
  \cS_m\ni A=(a_{i,j})_{1\leq i,j\leq m}\mapsto
    \Xi\otimes A=(a_{1,1},a_{2,1},a_{2,2},a_{3,1},\dots,a_{m,m})^\intercal
    \in\R^{m(m+1)/2},
\]
with $\otimes$ the matrix-tensor product in the first two components. With
this, we define two additional information state processes
$\Upsilon^u=(\Upsilon_t^u)_{t\in[0,T]}$ and $\Gamma^u=(\Gamma_t^u)_{t\in[0,T]}$
by
\[
  \Upsilon_t^u:=\int_0^t\frac{b}{\sigma^2}(s,Y_s^u,u_s)\rd Y_s^u
  \qquad\text{and}\qquad
  \Gamma_t^u:=\int_0^t\Xi\otimes\frac{bb^\intercal}{\sigma^2}(s,Y_s^u,u_s)\rd s,
\]
taking values in $\R^m$ and $\R^{m(m+1)/2}$, respectively. Note that the two
processes are related via $\langle\Upsilon^u\rangle=\Xi^{-1}\otimes\Gamma^u$.

Next, we define a function $F:\R^m\times\R^{m(m+1)/2}\to\R$ by
\begin{equation}\label{eq:def-filter-F}
  F(\upsilon,\gamma):=\int_\R\exp\Bigl(\ell^\intercal\upsilon
    -\frac{1}{2}\ell^\intercal (\Xi^{-1}\otimes\gamma)\ell\Bigr)\mu(\rd\ell)
\end{equation}
and, for any function $\phi:\R\to\R$, the transformation
$F[\phi]:\R^m\times\R^{m(m+1)/2}\to\R$ by
\begin{equation}\label{eq:def-filter-F-transform}
  F[\phi](\upsilon,\gamma):=\int_\R\phi(\ell)\exp\Bigl(\ell^\intercal\upsilon
    -\frac{1}{2}\ell^\intercal(\Xi^{-1}\otimes\gamma) \ell\Bigr)\mu(\rd\ell).
\end{equation}
We note that, by \cref{eq:def-hat-Z}, we can express $\hat{Z}^u$ in terms of
$F$ evaluated along $\Upsilon^u,\Gamma^u$, that is
$\hat{Z}^u=F(\Upsilon^u,\Gamma^u)$.

\begin{remark}
In \cite{ekstrom2022bayesian}, $F$ as defined in \cref{eq:def-filter-F} is
referred to as the \emph{Widder transform} of $\mu$ (due to
\cite{widder1944positive}, see also
\cite[Section 4.3 B]{karatzas1998brownian}). This should not be confused with
the Post--Widder transform common in the theory of Laplace transforms.
\end{remark}

For $u\in\cA$, we define the \emph{unnormalized conditional distribution} of
$\lambda$ under $\Q^u$ given $\cY^u$, denoted by
$\rho^u=(\rho_t^u)_{t\in[0,T]}$, as
\[
  \rho_t^u(A):=F[\mathds{1}_{A}](\Upsilon_t^u,\Gamma_t^u)
  =\int_A\exp\Bigl(\ell^\intercal\Upsilon_t^u
    -\frac{1}{2}\ell^\intercal(\Xi^{-1}\otimes \Gamma_t^u)\ell\Bigr)\mu(\rd\ell)
    \quad\forall A\in\cB(\R),
\]
where $\cB(\R)$ is the Borel $\sigma$-field over $\R$. Using Bayes' rule
\cite[Lemma 3.5.3]{karatzas1998brownian}, we can identify the \emph{normalized
conditional distribution} of $\lambda$, denoted by
$\pi^u=(\pi_t^u)_{t\in[0,T]}$, as
\[
  \pi_t^u(A):=\frac{\rho_t^u(A)}{\rho_t^u(\R)}
  =\frac{\E^{Q^u}[Z_t^u\mathds{1}_A(\lambda)|\cY^u_t]}{\E^{\Q^u}[Z_t^u|\cY^u_t]}
  =\P(\lambda\in A|\cY^u_t).
\]
Noticing that the gradient of $F$ with respect to $\upsilon$ is given by
\[
  F_\upsilon(\upsilon,\gamma)
  = \int_\R\ell\exp\Bigl(\ell^\intercal\upsilon
    -\frac{1}{2}\ell^\intercal(\Xi^{-1}\otimes\gamma)\ell\Bigr)\mu(\rd\ell)
  = F[\mathrm{id}](\upsilon,\gamma),
\]
we find that the conditional mean of $\lambda$ can be expressed in terms of the
process $m^u=(m_t^u)_{t\in[0,T]}$ given by
\begin{equation}\label{eq:relat-G-con-mean}
  m_t^u:=G(\Upsilon^u_t,\Gamma^u_t)=\E[\lambda|\cY^u_t]
    =\int_\R\ell\pi_t^u(\rd\ell),
\end{equation}
where $G:\R^m\times\R^{m(m+1)/2}\to\R$ is defined as
\begin{equation}\label{eq:def-G}
  G(\upsilon,\gamma):=\frac{F_\upsilon}{F}(\upsilon,\gamma).
\end{equation}
By \Cref{asspt-mu}, we find that $|G|$ is bounded by the same
constant $K>0$ as $\lambda$.

\begin{remark}
In the above definitions, we have always constructed continuous versions of the
conditional expectations, which are therefore measurable. The other identities
then hold in a $\P$-a.s.\ sense. 
\end{remark}

\begin{lemma}\label{lem:G-is-Lip}
The function $G$ is Lipschitz continuous.
\end{lemma}

\begin{proof}
Since $G$ is continuously differentiable, it suffices to show that its gradient
is uniformly bounded. For this, we note that
\[
  |\rD G|^2 =\Bigl|\frac{F_{\upsilon\upsilon}}{F}-G^2\Bigr|^2
    +\Bigl|\frac{F_{\upsilon\gamma}}{F}-\frac{F_\gamma}{F}G\Bigr|^2
	\leq 2\Bigl\{\Bigl|\frac{F_{\upsilon\upsilon}}{F}\Bigr|^2
    +\Bigl|\frac{F_\upsilon^2}{F^2}\Bigr|^2
	+\Bigl|\frac{F_{\upsilon\gamma}}{F}\Bigr|^2
    +\Bigl|\frac{F_\gamma F_\upsilon}{F^2}\Bigr|^2\Big\}.
\]
As $\mu$ is compactly supported by \Cref{asspt-mu}, for the mixed derivative we
obtain
\[
  |F_{\upsilon\gamma}(\upsilon,\gamma)|
    \leq \int_\R K^3\exp\Bigl(\ell^\intercal\upsilon
      -\frac{1}{2}\ell^\intercal\bigl(\Xi^{-1}\otimes\gamma\bigr)\ell\Bigr)
      \mu(\rd\ell)
    \leq K^3 F(\upsilon,\gamma),
\]
and similar estimates hold for the other terms. It follows that $G$ is
Lipschitz.
\end{proof}

Next, for all controls $u\in\cA$, we now define the corresponding
\emph{innovations process} $V^u=(V^u_t)_{t\in[0,T]}$ by
\begin{align}
  V^u_t &:= 
  \int_0^t\frac{1}{\sigma(s,Y_s^u,u_s)}\bigl(\rd Y_s^u-(m_s^u)^\intercal
    b(s,Y_s^u,u_s) \rd s\bigr)\label{eq:def-innov-process}\\
  &=W_t+\int_0^t(\lambda-m_s^u)^\intercal\frac{b}{\sigma}(s,Y_s^u,u_s)
    \rd s.\notag
\end{align}
The following key lemma can be found in
\cite[Lemma 11.3]{liptser2013statisticsII} or
\cite[Lemma 22.1.7]{cohen2015stochastic}.

\begin{lemma}
For all controls $u\in\cA$, the corresponding innovations process $V^u$ is a
standard $(\cY^u,\P)$-Brownian motion.
\end{lemma}

Using \cref{eq:relat-G-con-mean} and \cref{eq:def-innov-process}, we can rewrite
$Y^u$ with dynamics in the observation filtration $\cY^u$ as
\[
  Y_t^u=\int_0^tG(\Upsilon^u_s,\Gamma^u_s)^\intercal b(s,Y_s^u,u_s)\rd s
    +\int_0^t\sigma(s,Y_s^u,u_s)\rd V_s^u.
\]
Similarly, the first auxiliary state $\Upsilon^u$ can be written as
\[
  \Upsilon_t^u=\int_0^t\Bigl[\frac{b}{\sigma^2}
    (s,Y_s^u,u_s)G(\Upsilon^u_s,\Gamma^u_s)^\intercal b(s,Y_s^u,u_s)\Bigr]\rd s
  +\int_0^t\frac{b}{\sigma}(s,Y_s^u,u_s)\rd V_s^u.
\]
We now define the transformed running and terminal cost function $\tilde{k}$ and
$\tilde{g}$ as
\[
  \tilde{k}(t,y,\upsilon,\gamma,u)
    :=\frac{F[k(t,y,u,\argdot)](\upsilon,\gamma)}{F(\upsilon,\gamma)}
  \quad\text{and}\quad
  \tilde{g}(y,\upsilon,\gamma)
    :=\frac{F[g(y,\argdot)](\upsilon,\gamma)}{F(\upsilon,\gamma)},
\]
which are continuous as $k$ and $g$ are uniformly continuous in their last
argument by combining \Cref{asspt-mu} and \Cref{asspt-cost}. Indeed, with
$z:=(t,y,\upsilon,\gamma,u)$ and $z_n:=(t_n,y_n,\upsilon_n,\gamma_n,u_n)$ such
that $z_n\to z$, we immediately have
\begin{align*}
  F[k(t_n,y_n,u_n,\argdot)](\upsilon_n,\gamma_n)
    &= \int_\R k(t_n,y_n,u_n,\ell)\exp\Bigl(\ell^\intercal\upsilon_n
      -\frac{1}{2}\ell^\intercal(\Xi^{-1}\otimes\gamma_n)\ell\Bigr)
      \mu(\rd\ell)\\
    &\to F[k(s,y,u,\argdot)](\upsilon,\gamma)
\end{align*}
by dominated convergence. This is justified as $\mu$ is compactly supported, and
the integrand is jointly continuous and as a result bounded along
$(z_n)_{n\in\N}$. Thus, $\tilde{k}$ is continuous as a composition of continuous
functions. Taking the conditional expectation with respect to $\cY^u$ in the cost
functional and using Fubini's theorem, we obtain
\[
  \cJ(u)=\E\Bigl[\int_0^T\tilde{k}(t,Y_t^u,\Upsilon_t^u,\Gamma_t^u,u_t)\rd t
    +\tilde{g}(Y_T^u,\Upsilon_T^u,\Gamma_T^u)\Bigr].
\]
This is now a control problem under full information, but with control-dependent
noise and filtration, which we subsequently formulate dynamically. For ease of
notation, we let $\tilde{m}:=1+m+m(m+1)/2$.
Define the \emph{extended state
space} $\S:=[0,T]\times\R^{\tilde{m}}$ and, for $(t,x)\in\S$, the
$\tilde{m}$-dimensional \emph{extended state}
$X^{u;t,x}:=(A^{u;t,x},\Upsilon^{u;t,x},\Gamma^{u;t,x})$ as the solution of
\begin{equation}\label{eq:aug-state}
  \rd X^u_s
    =\begin{pmatrix}
      G(\Upsilon^u_s,\Gamma^u_s)^\intercal b\big(s,A_s^u,u_s)\\
      \frac{b}{\sigma^2}(s,A_s^u,u_s)G(\Upsilon^u_s,\Gamma^u_s)^\intercal
        b(s,A_s^u,u_s)\\
      \Xi\otimes\frac{bb^\intercal}{\sigma^2}(s,A_s^u,u_s)
	\end{pmatrix}\rd s
	+\begin{pmatrix}
	  \sigma(s,A_s^u,u_s)\\
        \frac{b}{\sigma}(s,A_s^u,u_s)\\
        0
	\end{pmatrix}\rd V_s^u
\end{equation}
for $s\in[t,T]$ with initial condition $X^u_t=x$. Moreover, we define the drift
and diffusion coefficient functions $f,\Sigma:\S\times \cU\to\R^{\tilde{m}}$
with $x=(a,\upsilon,\gamma)$ as
\[
  f(t,x,u):=\begin{pmatrix}
    G(\upsilon,\gamma)^\intercal b(t,a,u)\\
    \frac{b}{\sigma^2}(t,a,u)G(\upsilon,\gamma)^\intercal b(t,a,u)\\
    \Xi\otimes\frac{bb^\intercal}{\sigma^2}(t,a,u)
  \end{pmatrix}
  \quad\text{and}\quad
  \Sigma(t,x,u):=\begin{pmatrix}
    \sigma(t,a,u)\\
    \frac{b}{\sigma}(t,a,u)\\
    0
  \end{pmatrix}.
\]
For each fixed control $u\in\cA$, the coefficient functions
$f(\argdot,u),\Sigma(\argdot,u):\S\to\R^{\tilde{m}}$ are products of bounded,
Lipschitz continuous functions, so they are also Lipschitz continuous. Hence, by
a standard existence result (such as \cite[Theorem 2.5.7]{krylov2008controlled}),
for every initial condition $(t,x)\in\S$ and control $u\in\cA$ there exists a
pathwise unique $(\cY^u,\P)$-strong solution of \eqref{eq:aug-state}. Note that
the diffusion coefficient $\Sigma$ is degenerate, as the one-dimensional
innovations process is the only driving noise and the third state variable is not even
diffusive.

With this, we define the \emph{extended cost functional} as
\begin{equation}\label{eq:def:extended-cost-fucntional}
  \cJ(u;t,x):=\E\Bigl[\int_t^T\tilde{k}(s,X_s^{u;t,x},u_s)\rd s
    +\tilde{g}(X_{T}^{u;t,x})\Bigr]\qquad 
	\text{subject to }\cref{eq:aug-state}
\end{equation}
and the value function of the control problem as
\[
  V(t,x):= \inf_{u\in\cA}\cJ(u;t,x).
\]
The notion of ($\varepsilon$-)optimality given in \Cref{def:optimality} applies
for each $(t,x)\in\S$ in the obvious way.

\begin{remark}\label{rem:connection-aug-orig}
\begin{itemize}
\item[1)] The extended control problem includes the original optimization
problem \cref{eq:original-OP}. Specifically, for the solution $X^{u;0,0}$ of
\cref{eq:aug-state} we have $Y^u=A^{u;0,0}$ (as $Y^u_0=0$) for all $u\in\cA$,
hence $\cJ(u)=\cJ(u;0,0)$. This means that a $(0,0)$-optimal control for
the extended cost functional \cref{eq:def:extended-cost-fucntional} is also
optimal for the original cost functional \cref{eq:original-OP}, as the
minimization is over $\cA$ in both cases.
\item[2)] In general $V^u$ does not generate $\cY^u$ (this is the
\emph{innovations problem}, see e.g.\ \cite{heunis2011innovations}) but is only
adapted to it. However, we will see that, for the $\varepsilon$-optimal controls
$u^{\varepsilon}$ constructed below, the corresponding innovations process
generates the observation filtration, since the solution of the state equation
is strong and $\Sigma$ admits a left-inverse.
\end{itemize}
\end{remark}

\begin{remark}
Obtaining a finite dimensional description of the conditional distribution for a
\emph{time-dependent} hidden process $\lambda=(\lambda_t)_{t\in[0,T]}$ is known
only in two cases, first in the conditionally Gaussian case (see
\cite[Chapter 12]{liptser2013statisticsII}) and second for a finite-state Markov
chain (see \cite[Chapter 9]{liptser2013statisticsI}) for which sufficient
conditions for optimality are given in \cite{caines1985optimal} in the form of a
verification theorem. Since in our setting $\lambda$ is static, we can work with
a general distribution $\mu$ and still obtain an $(m+m(m+1)/2)$-dimensional
description of the conditional distribution due to the separable structure of
the drift.
\end{remark}

In the following, for any $x\in\R^{\tilde{m}}$, we write $x=(a,\upsilon,\gamma)$
with $a\in\R$, $\upsilon\in\R^m$, and $\gamma\in\R^{m(m+1)/2}$. With this, the
HJB equation for the extended control problem reads
\[
  \partial_tV+\inf_{u\in\cU}\{\cL^u V+\tilde{k}(\argdot,u)\}=0,
  \qquad V(T,\argdot)=\tilde{g},
\]
where for $u\in\cU,b=b(\argdot,u),\sigma=\sigma(\argdot,u)$ the
infinitesimal generator $\cL^u$ is given by
\begin{equation}\label{eq:inf_generator}
  \cL^u=G^\intercal b\,\partial_a
    +\frac{b}{\sigma^2}G^\intercal b\,\rD_\upsilon
	+\Bigl(\Xi\otimes\frac{bb^\intercal}{\sigma^2}\Bigr)\rD_\gamma
	+\frac{1}{2}\trace\Big[\Bigl(\sigma^2\rD_{aa} +\frac{bb^\intercal}{\sigma^2}
      \rD_{\upsilon\upsilon}
	+2b\rD_{a\upsilon}\Bigr)\Big].
\end{equation}
The HJB equation is fully nonlinear and degenerate as the second-order
coefficient matrix is always of rank one, thus $\cL^u$ is not uniformly
elliptic.

\begin{remark}\label{rem:HJB}
To the best of our knowledge, no explicit solutions of the HJB equation have
been
obtained beyond the special cases of
\cite{benevs1991separation,karatzas1992resolvent} (infinite horizon) and
\cite{karatzas1993finite} (finite horizon). Furthermore, there are no existence
results which yield a solution sufficiently regular to apply classical
verification. Such a classical verification theorem can nevertheless still be
proved under the usual regularity assumptions. In
\cite{benevs1991separation,karatzas1992resolvent,karatzas1993finite}, explicit
classical solutions to the respective HJB equations were obtained. They also
show
that the optimally controlled state process does not admit a strong solution.
\end{remark}

\subsection{Connection of control and higher-order moments}
\label{sec:Active-Learning-Aspect}

In this subsection we briefly elaborate on the connection of the control and
higher order conditional moments of $\lambda$. In particular, we draw
connections
to the conditional variance and the dual effect mentioned in the introduction.
For simplicity, calculations are presented for one-dimensional $\lambda$ only.
By definition of $F$, we see that with
\[
  G_k(\upsilon,\gamma):=\frac{\partial_\upsilon^k F}{F}(\upsilon,\gamma),
\]
the conditional moments of $k$-th order are given by
\[
  m^{k;u}_t:=G_k(\Upsilon^u_t,\Gamma^u_t)=\E[\lambda^k|\cY_t^u].
\]
As a consequence the conditional variance in the initial problem is given by
\[
  \mathrm{var}^u_t:=\V[\lambda|\cY_t^u]
  =G_2(\Upsilon^u_t,\Gamma^u_t)
    -G(\Upsilon^u_t,\Gamma^u_t)^{2},
\]
and straightforward computations detailed in \Cref{app:Active-Learning} show
that
\[
  \rd\text{var}^u_t = -\Bigl[G^2(G_2+G^2)\Bigr](\Upsilon^u_t,\Gamma^u_t)
    \frac{b^2}{\sigma^2}(t,Y^u_t,u_t)\rd t
    +G_{\upsilon\upsilon}(\Upsilon^u_t,\Gamma^u_t)
    \frac{b}{\sigma}(t,Y^u_t,u_t)\rd V^u_t.
\]
From this, we see that the controls exhibit the dual effect according to the
definition given in \cite{bar1974dual}. There, the control is said to have
\emph{no dual effect of order $k$} (or \emph{neutral} in the language of
\cite{feldbaum1960dual}), if all moments of higher order are independent of the
control in the sense that they agree with the conditional moments given $\cY^0$
with $u\equiv0$ (the ``inactive control"). The control is said to \emph{have a
dual effect}, if it affects any higher order moment. Here, calculations
similar to the above show that
\[
  \rd m^{2;u}_t = G_{2,\upsilon}(\Upsilon^u_t,\Gamma^u_t)
    \frac{b}{\sigma}(t,Y^u_t,u_t)\rd V^u_t,
\]
i.e.\ the second moment is generally control-dependent and the dual effect is
present.

\section{Weak formulation}\label{sec:Weak-OP}

In this section we formulate the extended control problem in its weak
formulation. We allow for the most general weak admissible controls in which the
underlying filtered probability space and Brownian motion are part of the
control, thereby introducing a control problem smaller in value than
\cref{eq:def:extended-cost-fucntional}. The problem is then transformed into one
under full information, exactly as in \Cref{sec:trans-FI}. The purpose of this
is to show that the values in strong and weak formulation agree and to allow
comparison of these approaches from a modeling perspective.

\begin{definition}
For $(t,x)\in\S$, a \emph{weak admissible control} is a seven-tuple
$U^{t,x}=(\Omega^{t,x},\mfF^{t,x},\cF^{t,x},\P^{t,x},W^{t,x},X^{t,x},u^{t,x})$
such that
\begin{itemize}
\item[1)] $(\Omega^{t,x},\mfF^{t,x},\P^{t,x})$ is a probability space and the
filtration $\cF^{t,x}$ satisfies the usual conditions;
\item[2)] $W^{t,x}$ is a one-dimensional standard
$(\cF^{t,x},\P^{t,x})$-Brownian motion;
\item[3)] $u$ is $\cF^{t,x}$-progressively measurable and $\cU$-valued;
\item[4)] $X^{t,x}=(A^{t,x},\Upsilon^{t,x},\Gamma^{t,x})$ is a continuous and
adapted process and the unique strong solution of the state equation
\cref{eq:aug-state} in $(\Omega^{t,x},\mfF^{t,x},\cF^{t,x},\P^{t,x},W^{t,x})$
with $X_t=x$, that is, the tuple
$(\Omega^{t,x},\mfF^{t,x},\cF^{t,x},\P^{t,x},W^{t,x},X^{t,x})$ is a weak
solution of the SDE with that initial condition.
\end{itemize}
All objects above are understood with time index set $[t,T]$. The set of all
weak admissible controls is denoted by $\cA^{weak}(t,x)$.
\end{definition}

With this, the value function over the set of weak controls is defined as
\[
  V^{weak}(t,x):=\inf_{U^{t,x}\in\cA^{weak}(t,x)}
    \E^{t,x}\Bigl[\int_t^T\tilde{k}(s,X_s^{t,x},u^{t,x}_s)\rd s
    +\tilde{g}(X_T^{t,x})\Bigr],
\]
where $\E^{t,x}$ is the expectation under $\P^{t,x}$. From the fact that the
state equation \cref{eq:aug-state} admits a strong solution, it follows that
the set of weak admissible controls is non-empty. As any strong solution of
\cref{eq:aug-state} is also a weak solution and the filtration is part of the
control in this formulation, it is clear that we can embed
$\cA\hookrightarrow\cA^{weak}(t,x)$, and hence
\begin{equation}\label{eq:V-V-weak}
  V \geq V^{weak}\qquad \text{on }\S.
\end{equation}
This is the first key inequality which allows us to bound $V$ from below by a
``nicer" problem, where there is no dependence of the filtration and noise on
the control.

\begin{remark}\label{rem:comparsion-of-formulations}
The case for a weak formulation, like the one in this subsection, is the
unobservability of the components of the state, in our case $\lambda,W$.
Specifically, we cannot identify the Brownian motion $W$ by observing only the
state $Y^u$ and, as a consequence, should be comfortable allowing the
distribution and the driving Brownian motion to vary as part of the control,
hence leading to the weak formulation. Furthermore, as long as the filtration
generated by the state is only extended by independent ``auxiliary" randomness,
this does not violate the information pattern of basing decisions only on
observations of the state.

The case against a weak formulation can also be made, as the noise process in
the form of Brownian motion is generally control-independent and given ``by
nature", i.e.\ it is fixed. Furthermore, in order to make theoretical use of
the above construction, one might have to work in a filtration strictly larger
than the filtration generated by the state process $Y^u$, which in a sense
violates a part of the idea behind the model, namely that the decision has to
be made only on the basis of the information generated by the state $Y^u$.
Another even more questionable point concerns the wide- and strict-sense
admissible controls considered in
\cite{benevs1991separation,karatzas1992resolvent,karatzas1993finite}. There,
the ``observation filtrations" to which the controls are adapted are required
to be independent of $\lambda$. But $\lambda$ is part of the dynamics of the
observable state, and thus should certainly \emph{not} be independent of the
observation filtration. In the strong formulation, $\cY^u$ is generally not
independent of $\lambda$ under $\P$.
\end{remark}

\section{Viscosity characterization and
			\texorpdfstring{$\boldsymbol\varepsilon$}{epsilon}-optimal controls}
\label{sec:DPP-eps-opt}

Since obtaining a classical solution of the HJB equation is out of reach, as
pointed out in \Cref{rem:HJB}, we consider solutions in the viscosity sense
(see \cite{crandall1992user}) instead. Recall that the HJB equation is given by
\begin{equation}\label{eq:HJB}\tag{HJB}
  \partial_tV+\inf_{u\in\cU}\{\cL^uV+\tilde{k}(\argdot,u)\}=0,
  \qquad V(T,\argdot)=\tilde{g}
\end{equation}
on $\S$, where the infinitesimal generator $\cL^u$ is defined in \cref{eq:inf_generator}. We show that the value functions in the strong
and weak formulation of the problem are equal to the unique viscosity solution
of \cref{eq:HJB} using the stochastic Perron method. Moreover, we construct
piecewise constant $\varepsilon$-optimal controls, which are also Markovian on a
time-discretized grid. This allows us to link the strong and weak formulation
in a clean way.

Our agenda for the remainder of this section is to apply a version of the
stochastic Perron method \cite{bayraktar2013stochastic}. More precisely, we
\begin{itemize}
\item prove a comparison principle for semicontinuous viscosity solutions of
the HJB equation;
\item using \cite{barles1991convergence}, show that the infimum of value
functions of a family of auxiliary control problems with piecewise constant
controls is a viscosity subsolution of the HJB equation;
\item show that the supremum of stochastic subsolutions in the weak
formulation with weak admissible controls is a viscosity supersolution.
\end{itemize}
We can then use the comparison principle and the auxiliary control problems to
sandwich the value function $V$ and show that it is itself a viscosity solution
of \cref{eq:HJB}.

\subsection{The comparison principle}\label{subsec:comparison-principle}

The comparison principle for the HJB equation \cref{eq:HJB} is a standard
result. The main difficulty in the proof consists in controlling the viscosity
sub- and supersolutions at infinity, which can be achieved by constructing a
strict classical subsolution which grows sufficiently fast. The proof of the
comparison principle is deferred to \Cref{app:comparison}.

\begin{theorem}\label{th:comparison-principle}
Let $U:\S\to\R$ be an upper semicontinuous viscosity subsolution and
$W:\S\to\R$ be a lower semicontinuous viscosity supersolution of \cref{eq:HJB} for which there exist $C,q>0$ such that
\[
  0\leq v(t,x)\leq C(1+|x|^q)\qquad \forall v\in\{U,W\}, (t,x)\in\S.
\]
If $U(T,\argdot)\leq W(T,\argdot)$ on $\R^{\tilde{m}}$, then $U\leq W$
everywhere on $\S$.
\end{theorem}

\subsection{The infimum of supersolutions}\label{subsec:infimum-supersolutions}

The most challenging step in our approach is the viscosity subsolution property
of the value function $V$ in the strong formulation. The main problem is that,
in general, the set of admissible controls is not closed under pasting. That is,
given two controls $u_1,u_2\in\cA$ and any $t\in[0,T]$, the pasted control
$u:=u_1\mathds{1}_{[0,t]}+u_2\mathds{1}_{(t,T]}$ can fail to be admissible.
Since closedness under pasting is fundamental for the DPP to be valid, it is
not immediately obvious if the value function $V$ in the strong formulation can
be linked to the HJB equation.

In a nutshell, our approach is based on the following two main ideas. First,
by using the stochastic Perron method, we do not have to work with the value
function directly, but can in fact resort to a sufficiently rich class of
approximating functions from above as long as their pointwise infimum is a
viscosity subsolution of the HJB equation. In what follows, this class of
approximating functions is chosen to be the set of value functions with
piecewise constant controls on a given time grid. The advantage of choosing
these approximating functions is that it is relatively easy to show
that they admit optimal controls in feedback form that are stable under
pasting, which allows us to mitigate the problem of not being able to paste
\emph{arbitrary} controls.

Nevertheless, working with piecewise constant controls in our setting is still
non-trivial as, in the strong formulation, the noise and filtration are still
control-dependent. However, since the cost functional only depends on the
distribution of the underlying noise and there are optimal controls in feedback
form, we can first study an auxiliary control problem with a fixed Brownian
motion with respect to a fixed filtration to construct optimizers and then
replace the driving Brownian motion and filtration with the appropriate
control-dependent innovations process and filtration.

\subsubsection{Piecewise constant controls}

The control problem with piecewise constant controls is formulated with respect
to the original Brownian motion $W$ on our probability space $(\Omega,\mfF,\P)$
and with respect to the filtration $\cF^W$ generated by $W$ and augmented by the
$\P$-nullsets.

For each $n\in\N$ let $\delta_n := T2^{-n}$ be the dyadic step size of
order $n$ and define the associated time and space-time grid
\[
  \T^n := \{k\delta_n : k=0,\dots,2^n\}
  \quad\text{and}\quad
  \S^n := \T^n\times\R^{\tilde{m}}.
\]
With this, the set of piecewise constant controls is given by
\begin{multline*}
  \cA^n := \bigl\{u : [0,T]\times\Omega\to\cU :
    u\text{ is }\cF^W\text{-progressively measurable and}\\
    \text{ constant on }
    ((k-1)\delta_n,k\delta_n]\text{ for all }k=1,\dots,2^n\bigr\}.
\end{multline*}
Observe that $\cA^n\subset\cA^{pre}$, but in general $\cA^n\not\subseteq\cA$
since we assume the piecewise constant controls to be $\cF^W$-progressive.
In any case, for $u\in\cA^n$ and $(t,x)\in\S$, the associated state process 
$\hat{X}^{u;t,x}=(\hat{A}^{u;t,x},\hat{\Upsilon}^{u;t,x},
\hat{\Gamma}^{u;t,x})$ given as the unique strong solution of
\begin{equation}\label{eq:pw-constant-state}
  \rd\hat{X}^{u;t,x}_s = f(s,\hat{X}^{u;t,x}_s,u_s)\rd s
    +\Sigma(s,\hat{X}^{u;t,x}_s,u_s)\rd W_s,\quad \hat{X}^{u;t,x}_t=x
\end{equation}
with driving noise $W$ is well-defined. With this, the cost functional for the
piecewise constant control problem is defined as
\[
 \hat{\cJ}(u;t,x):=\E\Bigl[\int_t^T \tilde{k}(s,\hat{X}^{u;t,x}_s,u_s)\rd s
   +\tilde{g}(\hat{X}^{u;t,x}_T)\Bigr]\qquad 
	\text{subject to }\cref{eq:pw-constant-state}
\]
with associated value function $\hat{V}^n:\S\to\R$ given by
\[
  \hat{V}^n(t,x) := \inf_{u\in\cA^n} \hat{\cJ}(u;t,x),\qquad (t,x)\in\S.
\]
By Theorem 3.2.2 in \cite{krylov2008controlled}, for each $t\in[0,T]$ fixed
the mapping $x \mapsto \hat{\cJ}(u;t,x)$ is continuous, uniformly with respect
to $u\in\cA^n$, implying that also $x\mapsto\hat{V}^n(t,x)$ is continuous. With
this and using the pseudo-Markov property for piecewise constant controls
established in \cite[Lemma 3.2.14]{krylov2008controlled} (see also
\cite{claisse2016pseudo} for a discussion of the importance of the
pseudo-Markov property), it follows from classical arguments that the piecewise
constant control problem satisfies the following version of the DPP,
see \Cref{subsec:dpp} for the proof.

\begin{proposition}\label{prop:dpp}
Let $(t,x)\in\S^n$ with $t<T$, and for each $u\in\cU$ denote by $\hat{X}^{u;t,x}$
the state process with constant control $u$. Then it holds that
\begin{equation}\label{eq:dpp}
  \hat{V}^n(t,x) = \inf_{u\in\cU}\E\Bigl[\int_t^{t+\delta_n}
    \tilde{k}(s,\hat{X}^{u;t,x}_s,u)\rd s
    + \hat{V}^n(t+\delta_n,\hat{X}^{u;t,x}_{t+\delta_n})\Bigr].
\end{equation}
\end{proposition}

The advantage of the DPP is that it
gives us a convenient way to construct optimal piecewise constant controls.

\begin{theorem}\label{th:DPP-aux-OC}
For each $n\in\N$, there exists a measurable function $U^*_n:\S\to\cU$ such
that for each $(t,x)\in\S$ the SDE
\[
  \rd\hat{X}^{*;t,x}_s
    = f\bigl(s,\hat{X}^{*;t,x}_s,U^*_n(s,\hat{X}^{*;t,x}_s)\bigr)\rd s
    + \Sigma\bigl(s,\hat{X}^{*;t,x}_s,U^*_n(s,\hat{X}^{*;t,x}_s)\bigr)\rd W_s
\]
with $\hat{X}^{*;t,x}_s = x$ for all $s\in[0,t]$ admits a unique strong solution and such that the control process
\[
  u^*_s := U^*_n(s,\hat{X}^{*;t,x}_s),\quad s\in[0,T]
\]
is admissible and optimal for the piecewise constant control problem, that is
\[
  u^*\in\cA^n\quad\text{and}\quad \hat{\cJ}^n(u^*;t,x) = \hat{V}^n(t,x).
\]
\end{theorem}

\begin{proof}
It is sufficient to construct $U^*_n$ on $\S^n$ and extend it as a piecewise
constant function to $\S$. In particular, this guarantees that the control
$u^*$ is indeed piecewise constant. Let therefore $t\in\T^n$ with $t<T$.
According to Corollary 3.2.8 in \cite{krylov2008controlled}, the mapping
\[
  (x,u) \mapsto \E\Bigl[\int_t^{t+\delta_n}
    \tilde{k}(s,\hat{X}^{u;t,x}_s,u)\rd s
    + \hat{V}^n(t+\delta_n,\hat{X}^{u;t,x}_{t+\delta_n})\Bigr]
\]
is continuous on $\R^{\tilde{m}}\times\cU$. We may therefore apply the
measurable selection result \cite[Theorem 2]{schal1974selection} to obtain
a measurable optimizer $U^*_n:\S^n\to\cU$ of the right-hand side of the DPP
\cref{eq:dpp}. Clearly, this function satisfies the desired properties.
\end{proof}

With our hands on an optimal feedback control for the piecewise constant
control problem, we can now draw the connection to the full information problem
in the strong formulation.

\begin{proposition}\label{prop:cost-relation}
Let $n\in\N$ and for $(t,x)\in\S$ let $u^*\in\cA^n$ be the optimal control for
$\hat{V}^n(t,x)$ constructed in \Cref{th:DPP-aux-OC}. Then $u^*\in\cA$ and
\[
  V(t,x) \leq \cJ(u^*;t,x) = \hat{\cJ}^n(u^*;t,x) = \hat{V}^n(t,x).
\]
\end{proposition}

\begin{proof}
Let us first observe that $u^*\in\cA^{pre}$ and denote by $\hat{X}
=(\hat{A},\hat{\Upsilon},\hat{\Gamma})$ the state process associated with
$u^*$. Since $u^*$ is given in terms of a measurable function of $\hat{X}$, it
follows that $u^*$ is $\cF^{\hat{X}}$-progressive. But
$\cF^{\hat{X}}=\cF^{\hat{A}}$ and hence $u^*\in\cA$. Finally, since the
cost functional depends on the underlying Brownian motion only through its
distribution, it follows that $\cJ(u^*;t,x) = \hat{\cJ}^n(u^*;t,x)$ from which
we conclude.
\end{proof}

\subsubsection{Convergence of the value functions}

Up to this point, we have solved the piecewise constant control
problem and argued that the constructed optimizer induces an admissible control
in the full information problem in the strong formulation. It remains to argue
that the value functions $\hat{V}^n$ converge to a viscosity subsolution of the
HJB equation. This, however, is a standard argument since the DPP induces a
monotone, consistent, and stable approximation scheme in the sense of
\cite{barles1991convergence}.

To make this precise, let us fix $n\in\N$ and subsequently write $\mfM(\S^n)$
for the space of real-valued measurable functions on $\S^n$. We
introduce the approximation scheme at level $n$ in terms of a mapping
$S(n,\argdot):\S^n\times\R\times\mfM(\S^n)\to\R$ given by
\[
  S(n,t,x,v,w) := v - \inf_{u\in\cU}\E\Bigl[\int_t^{t+\delta_n}
    \tilde{k}(s,\hat{X}^{u;t,x}_s,u)\rd s
    + w(t+\delta_n,\hat{X}^{u;t,x}_{t+\delta_n})\Bigr],
    \quad t<T
\]
and
\[
 S(n,T,x,v,w) := v - \tilde{g}(x).
\]
Observe that the restriction of the piecewise constant value function
$\hat{V}^n$ to $\S^n$ solves this scheme in the sense that
\begin{equation}\label{eq:solution-to-scheme}
    S\bigl(n,t,x,\hat{V}^n(t,x),\hat{V}^n\bigr) = 0,\quad (t,x)\in\S^n.
\end{equation}
Following Example 2 in \cite{barles1991convergence}, the scheme $S$ is
monotone, consistent, and stable and hence the relaxed limit $V^+:\S\to\R$ of
the value functions $\hat{V}^n$, $n\in\N$, given by
\[
  V^+(t,x):=\limsup_{\substack{\S^n\ni(s,y)\to(t,x)\\ n\to\infty}}
    \hat{V}^n(s,y)
\]
is an upper semicontinuous function and a viscosity subsolution
of the HJB equation by Theorem 2.1 in \cite{barles1991convergence}. Note
that $V^+\geq V$ since $\hat{V}^n\geq V$ for all $n\in\N$. We gather these
results in the following theorem.

\begin{theorem}
The relaxed limit $V^+:\S\to\R$ of the piecewise constant value functions
$\hat{V}^n$, $n\in\N$, is an upper semicontinuous viscosity subsolution of the
HJB equation satisfying $V^+(T,\argdot)=\tilde{g}$ and $V^+\geq V$.
Moreover, there exist $C,q>0$ such that
\[
  0\leq V^+(t,x)\leq C(1+|x|^q)\qquad \forall (t,x)\in\S.
\]
\end{theorem}

\subsection{The supremum of subsolutions}\label{subsec:supremum-subsolutions}

It remains to show that $V^{weak}$ is bounded from below by a viscosity
supersolution of the HJB equation.
In order to achieve this, we rely on the notion of stochastic subsolutions
associated with the weak control problem as formulated in \Cref{sec:Weak-OP}.
These stochastic subsolutions are constructed in a way which guarantees that
they are dominated by the value function $V^{weak}$, and their pointwise
maximum is a viscosity supersolution of the HJB equation. Since the arguments
leading to these results are standard and follow \cite{bayraktar2013stochastic}
very closely, we keep the exposition to a minimum.

\begin{definition}\label{def:Stoch-Sub}
The set of stochastic subsolutions of \eqref{eq:HJB}, denoted by $\cV^-$, is
the set of all lower semicontinuous functions $W:\S\to\R$ such that
\begin{itemize}
\item[(1)] there exist constants $C,q>0$ such that
\[
  W(T,x)\leq\tilde{g}(x)\quad\text{and}\quad 0\leq W(t,x)\leq
    C(1+|x|^q)\qquad\forall (t,x)\in\S;
\]
\item[(2)] for all $(t,x)\in\S$, any weak admissible control
$U^{t,x}\in\cA^{weak}$, and any pair of $\cF^{t,x}$-stopping times
$t\leq\tau\leq\rho\leq T$, we have
\[
  W(\tau,X^{t,x}_\tau) \leq \E^{t,x}\Bigl[\int_\tau^\rho\tilde{k}(s,X_s^{t,x},u^{t,x}_s)\rd s
  +W(\rho,X_\rho^{t,x})\Big|\cF^{t,x}_\tau\Bigr].
\]
\end{itemize}
\end{definition}

Since the function $W\equiv 0$ is clearly a stochastic subsolution, we see that
$\cV^-\neq\emptyset$. Moreover, the submartingale property and the terminal
inequality directly show that
\[
 W(t,x) \leq \E^{t,x}\Bigl[\int_t^T\tilde{k}(s,X_s^{t,x},u^{t,x}_s)\rd s
    +\tilde{g}(X_T^{t,x})\Bigr]
\]
for any weak control and hence $W\leq V^{weak}$. In particular, it follows that
the pointwise supremum $V^-$ of all stochastic subsolutions
\[
 V^{-}(t,x):=\sup_{W\in\cV^{-}}W(t,x),
\]
is dominated by $V^{weak}$ and hence finite. Finally, as in
\cite[Theorem 4.1]{bayraktar2013stochastic}, we obtain the following key result, the proof of which is given in \Cref{app:supremum-stochastic-subsolutions}.

\begin{theorem}\label{th:supremum-is-supersolution}
The supremum $V^-$ of the set of stochastic subsolutions is a lower
semicontinuous viscosity supersolution of the HJB equation
satisfying $V^-(T,\argdot) = \tilde{g}$ and $V^-\leq V^{weak}$.
\end{theorem}

\subsection{Viscosity characterization and
            \texorpdfstring{$\boldsymbol\varepsilon$}{epsilon}-optimal controls}

It remains to piece together the results of the previous subsections to arrive
at the main result of this article. Up to this point, we have argued that
\[
  V^- \leq V^{weak} \leq V \leq V^+,
\]
and $V^-,V^+$ are, respectively, viscosity super- and subsolutions of the HJB
equation. Using the comparison principle, we therefore find that all functions
above are in fact equal, and we have constructed $\varepsilon$-optimal controls.

\begin{theorem}
It holds that $V^-=V^{weak}=V=V^+$ is the unique continuous viscosity solution
of the HJB equation in the class of nonnegative functions of polynomial growth
with terminal value $\tilde{g}$. Moreover, for each $\varepsilon>0$ and
$(t,x)\in\S$, there exists $n\in\N$ such that $\hat{V}^n(t,x) \leq V(t,x) +
\varepsilon$, and hence the optimal control associated with $\hat{V}^n(t,x)$ is
$\varepsilon$-optimal for $V(t,x)$.
\end{theorem}

\begin{proof}
We have $V^- \leq V^{weak} \leq V \leq V^+$ by construction.
Moreover, $V^-$ and $V^+$ are, respectively, lower and upper semicontinuous
viscosity super- and subsolutions of the HJB equation satisfying
$V^-(T,\argdot)=\tilde{g}=V^+(T,\argdot)$. The comparison principle hence
applies, showing that $V^+\leq V^-$, yielding the viscosity characterization.
The existence of $\varepsilon$-optimal controls follows directly from the
convergence $\hat{V}^n\to V^+=V$ and \Cref{prop:cost-relation}.
\end{proof}

\begin{remark}
The approach used in this paper in fact gives a general way to construct
$\varepsilon$-optimal controls for other control problems using stability
of viscosity solutions and the stochastic Perron method.
\end{remark}

\section{Application to a robotics control problem}\label{sec:Appl}

In this section, we present a toy application of our control methodology to a
simple robotics control problem in a wind tunnel. The primary objective of this
problem is to design an efficient and adaptive control strategy for a robot
that is subjected to dynamic uncertain wind forces while moving on a horizontal
one-dimensional plane. The goal is to maintain the robot's position as close to
the center as possible while minimizing energy cost and adapting to the
uncertainty of the motor's efficacy in the wind tunnel.

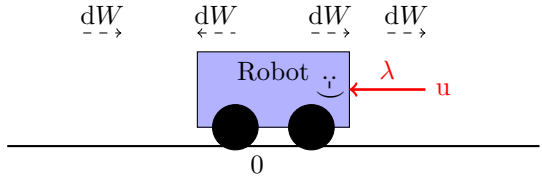
\begin{center}
\begin{tikzpicture}
  \draw[thick] (-2,0.25) -- (5,0.25);
  \node[below left] at (1.5, 0.25) {0};

  \draw[fill=blue!30] (0.5,0.5) rectangle (2.5,1.5);

  \draw[fill=black] (1,0.5) circle (0.3);
  \draw[fill=black] (2,0.5) circle (0.3);

  \node[below] at (1.5, 1.5) {Robot};
  \node[below, rotate=-90] at (2.55, 1.03) {:-)};

  \draw[dashed,->] (-1,1.75) -- (-0.5,1.75) node[midway,above] {$\rd W$};
  \draw[dashed,<-] (0.5,1.75) -- (1,1.75) node[midway,above] {$\rd W$};
  \draw[dashed,->] (2,1.75) -- (2.5,1.75) node[midway,above] {$\rd W$};
  \draw[dashed,->] (3,1.75) -- (3.5,1.75) node[midway,above] {$\rd W$};

  \draw[red,<-,line width=1pt] (2.5,1) -- (3.5,1) node[midway,above] 
  {$\lambda$} node[right] {u};
\end{tikzpicture}
\end{center}

The robot is newly built and the one-dimensional efficacy $\lambda$ of the
motor is uncertain in this environment. It is subject to wind $\rd W$ pushing
it back and forth on the one-dimensional plane. Only the position $Y^u$ of the
robot on the horizontal plane can be observed, in particular we cannot directly
observe the efficacy of the control $u$ through the motor or the wind $W$. As a
consequence, $\lambda$ must be estimated online from the position $Y^u$ of the
robot. The energy cost is taken into account quadratically (cost of control)
and the robot should be kept near the center. Deviation is penalized
quadratically during the task and at the end. We choose coefficient and cost
functions
\begin{align*}
  b(t,y,u) &=u, & \sigma(t,y,u) &= \sigma_0\\
  k(t,y,u,\ell) &= cy^2 + \rho u^2, & g(y,\ell) &= Cy^2,
\end{align*}
where the model parameters are given by
\[
  \sigma_0 = 1,\qquad
  T = 1,\qquad
  \rho = 2,\qquad
  c = 2,\qquad
  \text{and}\qquad
  C = 5.
\]
In the case of an observable efficacy $\lambda\in\R$, the problem reduces to a
standard stochastic linear-quadratic control problem which can be solved
explicitly up to the solution of a system of Riccati differential equations.
To be precise, the value function in the observable case is of the form
$V_{LQ}^{\lambda}(t,a)=f_1^\lambda(t)a^2+f_2^\lambda(t)$, where $f_1^\lambda,f_2^\lambda:[0,T]\to\R$ solve
\[
  0=\dot{f}_1^\lambda(t)-\frac{\lambda^2(f_1^\lambda)^2(t)}{\rho^2}+c
  \qquad\text{and}\qquad
  0=\dot{f}_2^\lambda(t)+f_1^\lambda(t)
\]
with terminal condition $f_1^\lambda(T)=C,f_2^\lambda(T)=0$. The feedback map
$u^\lambda_{LQ}:[0,T]\times\R\to\R$ for the optimal control in this problem is
given by
\[
  u^\lambda_{LQ}(t,a) =-\frac{\lambda\partial_aV_{LQ}^\lambda(t,a)}{2\rho}
    =-\frac{\lambda f_1^\lambda(t)}{\rho}a.
\]
The case of an unobservable $\lambda$ does not admit a closed-form solution and
has to be solved numerically. Here, we assume that $\lambda$ is uniformly
distributed over $[0,1]$ and compare the numerical approximation of an optimal
control for this problem with two benchmark controls obtained from the problem
with observable $\lambda$. The first one, the \emph{na\"{i}ve control}, is
constructed by replacing the random $\lambda$ by its mean
$\bar{\lambda}:=\E[\lambda]=0.5$, that is by considering the problem with
observable efficacy chosen as $\bar\lambda$. In other words, the na\"{i}ve
control in feedback form is given by
\[
  u^{naive}(t,a) := u^{\bar\lambda}_{LQ}(t,a) = -\frac{\bar\lambda f_1^{\bar\lambda}(t)}{\rho}a.
\]
The na\"{i}ve control does no updating of the estimate of $\lambda$ and thus
does not account for learning. The second benchmark control, the \emph{certainty
equivalent} (CE) control, is constructed by, at each time $t\in[0,T]$, acting
as if the conditional mean was the true $\lambda$, that is by replacing
$\lambda$ by its conditional mean $\E[\lambda|\cY^u_{t}]$ in the problem with
observable efficacy. Using the Markovian representation of the conditional mean
via $G$, the CE control is hence given
\[
  u^{CE}(t,x) := u^{G(\upsilon,\gamma)}_{LQ}(t,a).
\]
The CE control is built on the idea that the conditional mean is the best
approximation of $\lambda$ but ignores the effect of the control on higher
order moments, that is, it does not optimize for the dual effect. The expected
cost $V^{naive}$ and $V^{CE}$ associated with the two benchmark controls
$u^{naive}$ and $u^{CE}$ is computed by solving the linear PDE obtained by
plugging the benchmark controls into the HJB equation.

\subsection{Numerical implementation and results}

The controls and the associated expected costs are computed using a combination
of the deep Galerkin method (DGM) and policy iteration on the respective
PDEs. The choice of a deep learning method
over classical finite difference methods comes from the observation that the
latter methods are moderately inefficient due to the dimension of the
state space being equal to $1+3$.

Regarding the implementation of the DGM, let us highlight that we do not
approximate the value function directly, but rather approximate the value
function by $(t,x)\mapsto(T-t)V_\theta(t,x)+\tilde{g}(x)$ where $V_\theta$ is a
neural network parameterized by $\theta$. This directly embeds the terminal
condition into the approximating function. Second, we use the same DGM
architecture as suggested in \cite{sirignano2018dgm} with two layers for
$V_\theta$, and a simple two layer feedforward neural network for the
approximating control. Each sub-layer has 512 nodes. We use the Adam optimizer
with learning rate of $0.001$ for value function and control and alternate
gradient steps minimizing the infimum in the Hamiltonian and the DGM loss
functional in a 1:1 relation. We use batches of 7500 points and obtain a
terminal loss below $0.001$ after approximately $16\,000$ training epochs. As
an activation function, we use Sigmoid for both neural networks. The code is
available on
GitHub.\footnote{\scriptsize\url{https://github.com/AlexanderMerkel/Optimal-adaptive-control-with-separable-drift-uncertainty}}

\begin{figure}[H]
        \centering
        \tikzset{every picture/.style={scale=1.0}}
        \input{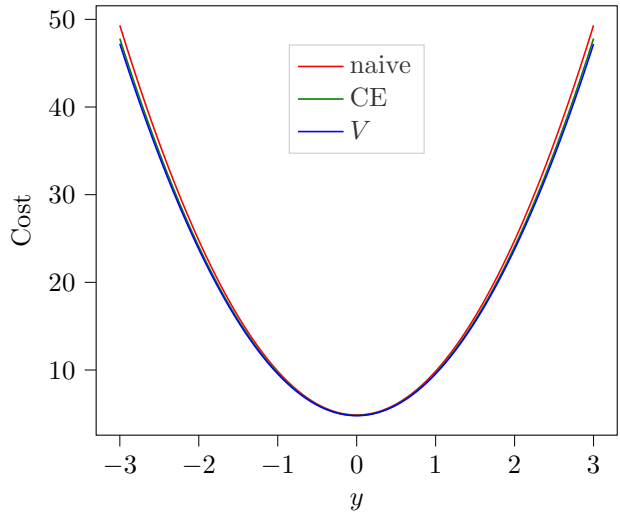}
        \caption{Cost vs.~state \( y \)}
        \label{fig:F1}
\end{figure}

\begin{figure}[H]
    \begin{subfigure}[t]{0.49\textwidth}
        \centering
        \tikzset{every picture/.style={scale=0.87}}
\begin{tikzpicture}

    \definecolor{darkgray176}{RGB}{176,176,176}
    \definecolor{green}{RGB}{0,128,0}
    \definecolor{lightgray204}{RGB}{204,204,204}
    
    \begin{axis}[
    legend cell align={left},
    legend style={
      fill opacity=0.8,
      draw opacity=1,
      text opacity=1,
      at={(0.03,1.10)},
      anchor=north west,
      draw=lightgray204
    },
    tick align=outside,
    tick pos=left,
    x grid style={darkgray176},
    xlabel={Conditional Variance},
    xmin=0.0601710110902786, xmax=0.0843926638364792,
    xtick style={color=black},
    y grid style={darkgray176},
    ylabel={Cost},
    ymin=8.93390545845032, ymax=9.33008971214294,
    ytick style={color=black}
    ]
    \addplot [semithick, red]
    table {%
    0.0612720251083374 9.13327598571777
    0.0613479614257812 9.13335800170898
    0.0614240169525146 9.13345623016357
    0.0615382194519043 9.13363456726074
    0.0616145133972168 9.13377094268799
    0.0617291927337646 9.13400268554688
    0.0618441104888916 9.13426971435547
    0.0619592666625977 9.13456535339355
    0.0621130466461182 9.13500785827637
    0.0622673034667969 9.13550281524658
    0.0624217987060547 9.13604736328125
    0.0625767707824707 9.13663768768311
    0.0627322196960449 9.13727378845215
    0.0629270076751709 9.13812637329102
    0.0631225109100342 9.13904476165771
    0.0633184909820557 9.14002227783203
    0.0635546445846558 9.14126777648926
    0.0637917518615723 9.1425895690918
    0.064069390296936 9.14421844482422
    0.0643482208251953 9.14593696594238
    0.0646284818649292 9.14773464202881
    0.0649904012680054 9.15015983581543
    0.0653139352798462 9.15241241455078
    0.0656797885894775 9.15505027770996
    0.0660886764526367 9.15810298919678
    0.0664588212966919 9.16095733642578
    0.0668309926986694 9.16390991210938
    0.0672886371612549 9.16765213012695
    0.0677913427352905 9.17190551757812
    0.0682129859924316 9.17558097839355
    0.0686371326446533 9.17937660217285
    0.0690639019012451 9.18328666687012
    0.0696655511856079 9.18894386291504
    0.0703589916229248 9.19564056396484
    0.0710588693618774 9.20252990722656
    0.074012279510498 9.23187637329102
    0.0749282836914062 9.24078941345215
    0.075621485710144 9.24742698669434
    0.0761799812316895 9.25268936157227
    0.0767418146133423 9.25789546966553
    0.0773544311523438 9.26345062255859
    0.0779234170913696 9.26848793029785
    0.0784478187561035 9.2730188369751
    0.0790231227874756 9.27786636352539
    0.0796018838882446 9.28261852264404
    0.0800865888595581 9.28652095794678
    0.080866813659668 9.29269409179688
    0.0819984674453735 9.30159950256348
    0.0826432704925537 9.30676460266113
    0.0832417011260986 9.31166648864746
    0.0832916498184204 9.3120813369751
    };
    \addlegendentry{naive}
    \addplot [semithick, green]
    table {%
    0.0612720251083374 9.03486442565918
    0.0613859891891479 9.03448963165283
    0.0615001916885376 9.03415489196777
    0.0616145133972168 9.03386402130127
    0.0617291927337646 9.03361415863037
    0.0618441104888916 9.03340530395508
    0.0619207620620728 9.03328704833984
    0.0620360374450684 9.03314304351807
    0.0621130466461182 9.0330696105957
    0.062190055847168 9.03301239013672
    0.0622673034667969 9.03297233581543
    0.0623444318771362 9.0329475402832
    0.0624217987060547 9.03293991088867
    0.0624992847442627 9.03294849395752
    0.0625767707824707 9.03297138214111
    0.0626544952392578 9.0330114364624
    0.0627322196960449 9.03306579589844
    0.062809944152832 9.03313541412354
    0.0629270076751709 9.03326606750488
    0.0630441904067993 9.0334300994873
    0.0631616115570068 9.03362560272217
    0.0632792711257935 9.03384971618652
    0.0633971691131592 9.03410530090332
    0.0635151863098145 9.03438663482666
    0.0636335611343384 9.03469562530518
    0.0637917518615723 9.0351505279541
    0.0639501810073853 9.03564834594727
    0.064109206199646 9.03619003295898
    0.0642684698104858 9.03677177429199
    0.0644282102584839 9.03739166259766
    0.0646284818649292 9.03821849822998
    0.0648293495178223 9.03909683227539
    0.0650308132171631 9.04002475738525
    0.0652732849121094 9.04120063781738
    0.065557599067688 9.04264831542969
    0.0658429861068726 9.04417133331299
    0.0661708116531372 9.04599380493164
    0.0666239261627197 9.04862594604492
    0.0670386552810669 9.05112552642822
    0.0675395727157593 9.05423736572266
    0.0682129859924316 9.05853652954102
    0.0691496133804321 9.06463432312012
    0.0702719688415527 9.07196807861328
    0.0712348222732544 9.07817935943604
    0.0721204280853271 9.08378982543945
    0.0730609893798828 9.08964920043945
    0.0754823684692383 9.10464859008789
    0.0760866403579712 9.10848617553711
    0.076977014541626 9.1142692565918
    0.0777808427810669 9.11959838867188
    0.0790712833404541 9.12817478179932
    0.0795536041259766 9.13129806518555
    0.0799894332885742 9.13404273986816
    0.080378532409668 9.13640975952148
    0.0807201862335205 9.1384162902832
    0.0810627937316895 9.14035224914551
    0.0814065933227539 9.14221382141113
    0.0817021131515503 9.1437520980835
    0.0819984674453735 9.14523506164551
    0.0823452472686768 9.14690208435059
    0.0826929807662964 9.14850902557373
    0.0830918550491333 9.15028762817383
    0.0832916498184204 9.15116214752197
    };
    \addlegendentry{CE}
    \addplot [semithick, blue]
    table {%
    0.0612720251083374 8.95191383361816
    0.0613859891891479 8.95208358764648
    0.0615001916885376 8.95227432250977
    0.0616527795791626 8.95256614685059
    0.0617674589157104 8.95280933380127
    0.0619207620620728 8.95316600799561
    0.0621130466461182 8.95366191864014
    0.0622673034667969 8.9540958404541
    0.0624217987060547 8.95455932617188
    0.0626544952392578 8.95531272888184
    0.0628490447998047 8.95598793029785
    0.0631225109100342 8.95699882507324
    0.0633971691131592 8.95808029174805
    0.0635941028594971 8.95889091491699
    0.0639106035232544 8.96024703979492
    0.0643084049224854 8.96203517913818
    0.0647087097167969 8.96390914916992
    0.0651519298553467 8.96605491638184
    0.0657205581665039 8.96889591217041
    0.0662117004394531 8.97141456604004
    0.0668309926986694 8.97466564178467
    0.06737220287323 8.97757339477539
    0.0679597854614258 8.98080348968506
    0.0684248208999634 8.98341655731201
    0.0688929557800293 8.98609733581543
    0.0694501399993896 8.98935508728027
    0.0700114965438843 8.99270725250244
    0.070708155632019 8.9969539642334
    0.0714552402496338 9.00158882141113
    0.0724327564239502 9.00770854949951
    0.073106050491333 9.0119047164917
    0.0736939907073975 9.01551628112793
    0.0741033554077148 9.01798629760742
    0.0746065378189087 9.02095222473145
    0.0751125812530518 9.02384090423584
    0.0755287408828735 9.02613830566406
    0.0758073329925537 9.02763557434082
    0.0762265920639038 9.02982330322266
    0.0765541791915894 9.03147983551025
    0.076977014541626 9.03355407714844
    0.0774490833282471 9.03579521179199
    0.0782568454742432 9.03951072692871
    0.0793120861053467 9.04435062408447
    0.0798439979553223 9.04688930511475
    0.0802323818206787 9.04882049560547
    0.080573558807373 9.05058097839355
    0.080866813659668 9.05214691162109
    0.0812591314315796 9.0543212890625
    0.0816035270690918 9.05630111694336
    0.0820974111557007 9.05923843383789
    0.0832916498184204 9.06647396087646
    };
    \addlegendentry{$V$}
    \end{axis}
    
    \end{tikzpicture}
    
        \caption{Cost vs.~conditional variance}
        \label{fig:F2}
    \end{subfigure}
    \begin{subfigure}[t]{0.49\textwidth}
        \centering
        \tikzset{every picture/.style={scale=0.87}}
        \input{VF_CM.tex}
        \caption{Cost vs.~conditional mean}
        \label{fig:F3}
    \end{subfigure}
    \label{fig:composite}
\end{figure}
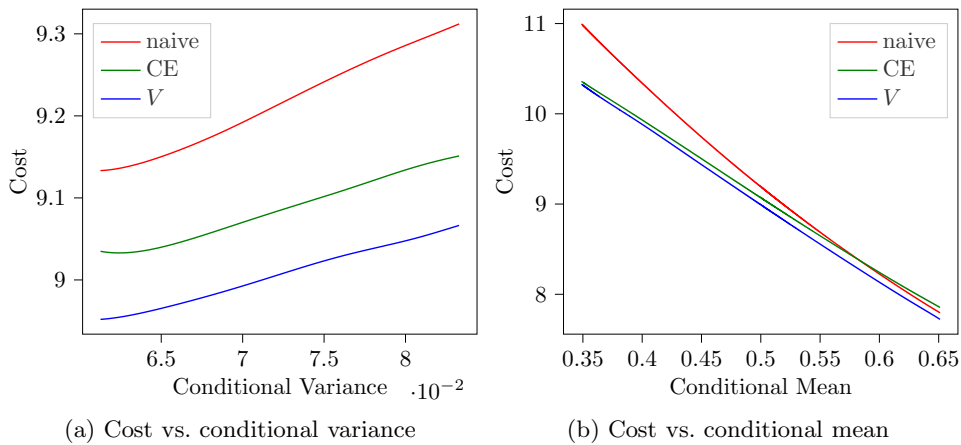
\Cref{fig:F1} compares the expected cost as functions of the initial state $y$,
for fixed time and auxiliary states $(t,\upsilon,\gamma)=(0,0,0)$. It shows a
small difference between the cost of the adaptive control and the cost induced
by the na\"ive and CE control, with the adaptive control leading to the smallest
total cost overall.

In \Cref{fig:F2}, we fix time $t=0.1$ and state $y=1$, and a target
conditional variance of $0.07$. We then identify pairs $(\upsilon,\gamma)$ such
that $G_\upsilon(\upsilon,\gamma)\approx 0.07$ and plot the expected cost with
$(t,y)$ fixed as a function of the conditional mean $G(\upsilon,\gamma)$. A
similar process is used for \Cref{fig:F3}, where we plot the expected cost as a
function of the conditional variance with a target conditional mean of $0.52$.
The numerical results show that there is a \emph{substantial} difference in the control actions and resulting costs, thus suggesting a significant advantage of using the adaptive control over the na\"{i}ve and CE control.

More precisely, \Cref{fig:F2} and \Cref{fig:F3} show that the CE control and
the adaptive control outperform the na\"ive control. We furthermore observe in
\Cref{fig:F2} that the expected cost is decreasing in conditional mean (which
is expected, as we plot for state $y=1$). More significantly, we see in 
\Cref{fig:F3} that the adaptive control shows a substantial difference compared
to the CE control. Moreover, for all three controls the expected cost is
increasing in the conditional variance, illustrating the failure of the
separation principle in this context.

\section*{Acknowledgements} This research was supported by the Deutsche Forschungsgemeinschaft through the Berlin--Oxford IRTG 2544: Stochastic Analysis in Interaction. SC also acknowledges the support of the UKRI Prosperity Partnership Scheme (FAIR) under EPSRC Grant EP/V056883/1, the Alan Turing Institute, and the Oxford--Man Institute for Quantitative Finance.

\appendix
\section{Appendix}\label{sec:Appendix}

\subsection{Computations of Section~\ref{sec:Active-Learning-Aspect}}
\label{app:Active-Learning}

In this appendix we elaborate on the calculations of
Section~\ref{sec:Active-Learning-Aspect} on the effect of the control of
higher order conditional moments of $\lambda$. Recall that, for simplicity, the
calculations are performed for the one-dimensional case $m=1$.
To begin with, we use It\^{o}'s formula to compute
\begin{align*}
  \rd m^u_t = \rd G(\Upsilon^u_t,\Gamma^u_t)
  &= G_\upsilon(\Upsilon^u_t,\Gamma^u_t)\rd\Upsilon^u_t
    +G_\gamma(\Upsilon^u_t,\Gamma^u_t)\rd\Gamma^u_t
    +\frac{1}{2}G_{\upsilon\upsilon}(\Upsilon^u_t,\Gamma^u_t)\rd 
      \langle\Upsilon^u\rangle_t\\
  &= \frac{b^2}{\sigma^2}(t,Y^u_t,u_t)
	\Bigl[G_\upsilon G+G_\gamma+\frac{1}{2}G_{\upsilon\upsilon}\Bigr]
    (\Upsilon^u_t,\Gamma^u_t)\rd t\\
  &\hspace{4cm}+G_\upsilon(\Upsilon^u_t,\Gamma^u_t)
    \frac{b}{\sigma}(t,Y^u_t,u_t)\rd V^u_t\\
  &= G_\upsilon(\Upsilon^u_t,\Gamma^u_t)\frac{b}{\sigma}(t,Y^u_t,u_t)\rd V^u_t,
\end{align*}
where the drift term indeed vanishes as the following calculation shows. First,
note that
\begin{align*}
  G_\upsilon &= \frac{F_{\upsilon\upsilon}F-F_{\upsilon}^2}{F^2},
  &
  G_\upsilon G &=\frac{F_{\upsilon\upsilon}F_{\upsilon}F-F_{\upsilon}^3}{F^3},\\
  G_{\upsilon\upsilon} &= \frac{F_{\upsilon\upsilon\upsilon}F^2-
    3F_{\upsilon\upsilon}F_{\upsilon}F+2F_{\upsilon}^3}{F^3},
  &
  G_\gamma &= \frac{F_{\upsilon\gamma}F-F_{\upsilon}F_{\gamma}}{F^2}.
\end{align*}
With this, it follows that
\begin{align*}
  &\mathrel{\phantom{=}}G_\upsilon G+\frac{1}{2}G_{\upsilon\upsilon}+G_\gamma\\
  &=\frac{1}{F^3}\Bigl(F_{\upsilon\upsilon}F_{\upsilon}F-F_{\upsilon}^3
    +\frac{1}{2}\Bigl[F_{\upsilon\upsilon\upsilon}F^2-
      3F_{\upsilon\upsilon}F_{\upsilon}F+2F_{\upsilon}^3\Bigr]
    +F_{\upsilon\gamma}F^2-F_{\upsilon}F_{\gamma}F\Bigr)\\
  &= \frac{1}{F}\Bigl(\frac{1}{2}F_{\upsilon\upsilon\upsilon}
    +F_{\upsilon\gamma}\Bigr) - \frac{F_\upsilon}{F^2}
    \Bigl(\frac{1}{2}F_{\upsilon\upsilon}+F_\gamma\Bigr).
\end{align*}
A direct computation shows that $F$ satisfies the backward heat equation
$F_\gamma=-\frac{1}{2}F_{\upsilon\upsilon}$, and hence also 
$F_{\upsilon\gamma}=-\frac{1}{2}F_{\upsilon\upsilon\upsilon}$ by
differentiating once with respect to $\upsilon$. From this, we conclude that
\[
  G_\upsilon G+\frac{1}{2}G_{\upsilon\upsilon}+G_\gamma = 0,
\]
that is, the drift term indeed vanishes.

Next, let us compute the effect of the control on the variance. We have
\begin{align*}
  \rd\text{var}^u_t &= \rd G_\upsilon(\Upsilon^u_t,\Gamma^u_t)\\
  &= G_{\upsilon\upsilon}(\Upsilon^u_t,\Gamma^u_t)\rd\Upsilon^u_t
    +G_{\upsilon\gamma}(\Upsilon^u_t,\Gamma^u_t)\rd\Gamma^u_t
    +\frac{1}{2}G_{\upsilon\upsilon\upsilon}(\Upsilon^u_t,\Gamma^u_t)
      \rd\langle\Upsilon^u\rangle_t\\
  &= \Bigl[G_{\upsilon\upsilon}G+G_{\upsilon\gamma}
    +\frac{1}{2}G_{\upsilon\upsilon\upsilon}\Bigr]
    (\Upsilon^u_t,\Gamma^u_t)\frac{b^2}{\sigma^2}(t,Y^u_t,u_t)\rd t\\
  &\hspace{5cm}+G_{\upsilon\upsilon}(\Upsilon^u_t,\Gamma^u_t)
    \frac{b}{\sigma}(t,Y^u_t,u_t)\rd V^u_t.
\end{align*}
Similarly to the calculations above, using the backward heat equation multiple
times, we find that
\begin{align*}
  &\mathrel{\phantom{=}}G_{\upsilon\upsilon}G+G_{\upsilon\gamma}
    +\frac{1}{2}G_{\upsilon\upsilon\upsilon}\\
  &= \frac{1}{2F^{4}}\Bigl[
    -2F_{\upsilon\upsilon\upsilon}F_{\upsilon}F^2
    +F_{\upsilon\upsilon\upsilon\upsilon}F^3
    +2F_{\upsilon\upsilon\gamma}F^3
    -2F_{\upsilon\upsilon}F_\gamma F^2
    +4F_\upsilon^2F_\gamma F\\
  &\hspace{6cm}-4F_\upsilon F_{\upsilon\gamma}F^2
    -3F_{\upsilon\upsilon}^2F^2
    +6F_{\upsilon\upsilon}F_\upsilon^2F-2F_\upsilon^4\Bigr]\\
  &=\frac{1}{2F^{4}}\Bigl[4F_\upsilon^2 F_\gamma F-2F_\upsilon^4\Bigr]
    =-G^2(G_2+G^2),
\end{align*}
and we conclude that
\[
  \rd\text{var}^u_t = -\Bigl[G^2(G_2+G^2)\Bigr](\Upsilon^u_t,\Gamma^u_t)
    \frac{b^2}{\sigma^2}(t,Y^u_t,u_t)\rd t
    +G_{\upsilon\upsilon}(\Upsilon^u_t,\Gamma^u_t)
    \frac{b}{\sigma}(t,Y^u_t,u_t)\rd V^u_t.
\]

\subsection{Proof of \Cref{th:comparison-principle}}\label{app:comparison}

This appendix is dedicated to the proof of the comparison principle
\Cref{th:comparison-principle}. A key ingredient in the proof is the existence
of a strict classical subsolution of the HJB equation in the sense of the
following definition.

\begin{definition}\label{def:strict-classical-subsolution}
We say that a function $\psi\in\cC^{1,2}(\S)$ is a strict classical subsolution
of \eqref{eq:HJB}, if there exists a continuous function
$\kappa:\S\rightarrow(0,\infty)$ such that
\[
  \partial_t\psi(t,x)+\inf_{u\in\cU}\{\cL^u\psi(t,x)+\tilde{k}(t,x,u)\}
  \leq -\kappa(t,x) < 0,
  \quad \psi(T,x)\leq -\kappa(T,x)<0,
\]
for all $(t,x)\in\S$.
\end{definition}

The first step in the proof of the comparison principle consists of establishing
the existence of a strict classical subsolution growing sufficiently fast at infinity.

\begin{lemma}\label{lem:strict-classical-subsolution}
Let $\tilde{q}\geq 2$. There exist $\zeta,M>0$ such that
$\psi:\S\to(-\infty,0]$ defined as
\begin{equation}\label{eq:def-strict-class-sub}
	\psi(t,x):=-|x|^{\tilde{q}}\exp(\zeta(T-t))-M(1+T-t),\qquad (t,x)\in\S,
\end{equation}
is a strict classical subsolution of \eqref{eq:HJB}.
\end{lemma}
\begin{proof}
Regularity is clear. For any $(t,x)\in\S$, the derivatives of $\psi$ can be
computed explicitly and are given by
\begin{align*}
	\partial_{t}\psi(t,x)&=\zeta |x|^{\tilde{q}}\exp(\zeta(T-t))+M,\\
	\rD_{x}\psi(t,x)&=-\tilde{q}x|x|^{\tilde{q}-2}\exp\bigl(\zeta(T-t)\bigr),\\
	\rD_{xx}\psi(t,x)&=-\tilde{q}|x|^{\tilde{q}-4}\bigl[(\tilde{q}-2)xx^\intercal
      +\diag\bigl(|x|^{2}\bigr)\bigr]\exp\bigl(\zeta(T-t)\bigr).
\end{align*}
Note that the running cost $\tilde{k}$ is non-negative, so it suffices to show that there exists a continuous $\kappa:\S\to(0,\infty)$ such that
\[
	-\partial_{t}\psi-\inf_{u\in \cU}\cL^{u}\psi\leq -\kappa\quad \text{on }\S.
\]
By Cauchy--Schwarz and using that $f$ is bounded (by boundedness of $b,\sigma^{-1}$, and $G$), there exists $C_1>0$ such that
\[
	-f(t,x,u)^{\intercal}\rD_{x}\psi(t,x)
	\leq |f(t,x,u)||\rD_{x}\psi(t,x)|
	\leq C_1|x|^{\tilde{q}-1}\exp(\zeta(T-t)).
\]
Similarly, again by Cauchy--Schwarz and boundedness of $\Sigma$, there exists $C_2>0$ such that
\begin{multline*}
	-\frac{1}{2}\trace\bigl[\Sigma(t,x,u)\Sigma(t,x,u)^{\intercal}\rD_{xx}\psi(t,x)\bigr]\\
	\leq \frac{1}{2}||\Sigma(t,x,u)||^{2}||\rD_{xx}\psi(t,x)||
	\leq C_2|x|^{\tilde{q}-2}\exp(\zeta(T-t)).
\end{multline*}
Putting this together yields the existence of $N_{1},N_{2},N_{3}>0$ independent of $\zeta,M$ such that
\begin{align*}
	-\cL^{u}\psi(t,x)
	\leq \bigl(N_{1}|x|^{2}+N_{2}|x|+N_{3}\bigr)|x|^{\tilde{q}-2}\exp\bigl(\zeta(T-t)\bigr).
\end{align*}
Defining
\[
	\kappa(t,x):=M-\big((N_{1}-\zeta)|x|^{2}+N_{2}|x|+N_{3}\big)|x|^{\tilde{q}-2}\exp\bigl(\zeta(T-t)\bigr),\quad (t,x)\in\S
\]
and choosing $\zeta>N_{1}$, the set $\{x\in\R^{\tilde{m}}:(N_{1}-\zeta)|x|^{2}+N_{2}|x|+N_{3}\geq0\}$ is bounded. Hence $-\kappa$ is bounded from above and choosing $M$ large enough yields
\[
	-\partial_{t}\psi(t,x)-\inf_{u\in \cU}\cL^{u}\psi(t,x)
	\leq -\kappa(t,x)<0
\]
as desired.
\end{proof}

To proceed, we need to introduce some additional notation. In what follows,
we express the HJB
equation in terms of the function
$\rF:\S\times\R\times\R^{\tilde{m}}\times\cS_{\tilde{m}}\to\R$ given by
\[
 \rF(t,x,p,q,Q)\\:=-p-\inf_{u\in \cU}
\Bigl\{f(t,x,u) q+\frac{1}{2}\trace\bigl[(\Sigma(t,x,u)\Sigma(t,x,u)^{\intercal}Q\bigr]+\tilde{k}(t,x,u)\Bigr\}
\]
for all $(t,x,p,q,Q)\in \S\times\R\times\R^{\tilde{m}}\times\cS_{\tilde{m}}$.
Next, we write $\intS:=[0,T)\times\R^{\tilde{m}}$ for
the parabolic interior of the state space. With this notation in place, we
subsequently agree that when we speak of a viscosity subsolution of 
\cref{eq:HJB} we mean a function $w:\S\to\R$ which satisfies
of
\[
  F(\argdot,\partial_t w,\rD_x w,\rD_{xx} w) \leq 0\qquad\text{on }\intS
\]
in the viscosity sense and satisfies $w(T,\argdot)\leq \tilde{g}$ on
$\R^{\tilde{m}}$. Similarly, we say that $w$ is a $\kappa$-strict viscosity
subsolution
if
\[
  F(\argdot,\partial_t w,\rD_x w,\rD_{xx} w) \leq -\kappa\qquad\text{on }\intS
\]
in the viscosity sense for some strictly positive continuous function
$\kappa:\S\to(0,\infty)$ (without any requirements on the behavior of $w$ at
time $T$). We use analogous conventions for (strict) viscosity supersolutions.

Finally, we denote by $LSC(\S)$ and $USC(\S)$ the sets of real-valued lower
and upper semicontinuous functions on $\S$, respectively. With this, we
can formulate the following perturbation result. The result is well-known
and can e.g.\ be found in \cite{ishii1993viscosity} in a slightly different
context; we nevertheless present
a proof to keep the paper self-contained.

\begin{lemma}\label{lem:perturb}
Fix $\rho>1$, let $U\in USC(\S)$, $\psi$ and $\kappa$ as in
\Cref{lem:strict-classical-subsolution}, and define the perturbation
\begin{equation}\label{eq:def-perturb-fcts}
  U^{\rho}:=\frac{\rho-1}{\rho}U+\frac{1}{\rho}\psi.
\end{equation}
If $U$ is a viscosity subsolution of \eqref{eq:HJB}, then $U^{\rho}$ is a
$\kappa/\rho$--strict viscosity subsolution.
\end{lemma}

\begin{proof}
Fix $(t,x)\in\S$ and let $\varphi^{\rho}\in\cC^{1,2}(\S)$ be a test function
for the subsolution property of $U^{\rho}$ at $(t,x)$, that is
$U^{\rho}\leq \varphi^{\rho}$ on $\intS$ and $U^{\rho}(t,x)=\varphi^{\rho}(t,x)$. By definition of $U^{\rho}$ we find that
\[
  U\leq\varphi := \frac{\rho}{\rho-1}\varphi^{\rho}-\frac{1}{\rho-1}\psi
  \qquad\text{and}\qquad
  U(t,x)=\varphi(t,x),
\]
from which we obtain that $\varphi$ is a test function for the subsolution
property of $U$ and thus
\begin{equation}\label{eq:perturbation_proof}
  \rF\bigl(t,x,\partial_{t}\varphi(t,x),
    \rD_{x}\varphi(t,x),\rD_{xx}\varphi(t,x)\bigr)\leq 0.
\end{equation}
Using the definition of $\varphi$ one more time and the fact that $\psi$
is a strict classical subsolution of \cref{eq:HJB}, it follows that
\begin{align*}
  &\rF\bigl(t,x,\partial_{t}\varphi^{\rho}(t,x),
    \rD_{x}\varphi^{\rho}(t,x),\rD_{xx}\varphi^{\rho}(t,x)\bigr)\\
  &\hspace{2cm}\leq\frac{\rho-1}{\rho}\rF\bigl(t,x,\partial_{t}\varphi(t,x),
    \rD_x\varphi(t,x),\rD_{xx}\varphi(t,x)\bigr)\\
  &\hspace{5cm} +\frac{1}{\rho}\rF\bigl(t,x,\partial_{t}\psi(t,x),
    \rD_x\psi(t,x),\rD_{xx}\psi(t,x)\bigr)\\
  &\hspace{2cm}\leq\frac{\rho-1}{\rho}\rF\bigl(t,x,\partial_{t}\varphi(t,x),
    \rD_x\varphi(t,x),\rD_{xx}\varphi(t,x)\bigr)-\frac{1}{\rho}\kappa(t,x).
\end{align*}
Rearranging terms and using \cref{eq:perturbation_proof} yields
\begin{multline*}
  \rF\bigl(t,x,\partial_{t}\varphi^{\rho}(t,x),
    \rD_x\varphi^{\rho}(t,x),\rD_{xx}\varphi^{\rho}(t,x)\bigr)
    +\frac{1}{\rho}\kappa(t,x)\\
  \leq \frac{\rho-1}{\rho}\rF\bigl(t,x,\partial_{t}\varphi(t,x),
    \rD_x\varphi(t,x),\rD_{xx}\varphi(t,x)\bigr) \leq 0
\end{multline*}
which concludes the proof.
\end{proof}

We are now ready for the proof of the comparison principle.

\begin{proof}[Proof of \Cref{th:comparison-principle}]
Fix $\tilde{q}\geq 2$ such that $\tilde{q}>q$ and let $\psi$ be an associated
strict classical subsolution as constructed in
\Cref{lem:strict-classical-subsolution}. Next, let $\rho>1$ and denote by
$U^\rho$ the perturbation of $U$ in terms of $\rho$ and $\psi$ as in
\Cref{lem:perturb}. We proceed to show that $U^\rho\leq W$, implying that
$U\leq W$ by sending $\rho\to\infty$. We argue by contradiction by
assuming that there exists $(t^*,x^*)\in\S$ with
\begin{equation}\label{eq:proof-CP-contr-asspt}
  U^{\rho}(t^*,x^*)-W(t^*,x^*)>0.
\end{equation}
For all $k\in\N$, we then define a function
\[
  \varphi_k(t,x,\hat{x}) := U^{\rho}(t,x)-W(t,\hat{x})
    -\frac{k}{2}|x-\hat{x}|^{2}
\]
on the domain $\mfS:=[0,T]\times\R^{\tilde{m}}\times\R^{\tilde{m}}$ and set
\[
  \Theta_k:=\sup_{(t,x,\hat{x})\in\mfS} \varphi_{k}(t,x,\hat{x})
  \qquad\text{and}\qquad
  \Theta:=\sup_{(t,x)\in\S}\varphi_{0}(t,x,x).
\]
Using \cref{eq:proof-CP-contr-asspt}, we find that
\begin{equation}\label{eq:theta_positive}
  0<U^{\rho}(t^{*},x^{*})-W(t^{*},x^{*})
    \leq \Theta\leq \Theta_{k+1}\leq \Theta_{k}\leq \Theta_{0},\qquad k\in\N.
\end{equation}
Next, using that $U$ and $W$ are non-negative followed by the growth assumption on $U$ and finally $\tilde{q}>q$, we obtain
\[
  \Theta_{0}\leq\sup_{(t,x)\in\S}\Bigl\{U(t,x)+\frac{1}{\rho}\psi(t,x)\Bigr\}
    \leq \sup_{x\in\R^{\tilde{m}}}\bigg\{K(1+|x|^{q})
      -\frac{1}{\rho}|x|^{\tilde{q}}\bigg\}
	<\infty.
\]
From \cref{eq:theta_positive}, we see that any maximizing sequence for any
$\Theta_{k}$ must eventually be contained in the set
\begin{align*}
  A:=\bigl\{(t,x,\hat{x})\in\mfS:U^{\rho}(t,x)-W(t,\hat{x})\geq 0\bigr\}.
\end{align*}
The set $A$ is bounded and, by upper semicontinuity of $U^{\rho}-W$, closed and
thus compact. This guarantees the existence of a maximizer
$(t_{k},x_{k},\hat{x}_{k})\in A$ for $\Theta_{k}$. By compactness, we may
assume without loss of generality that the sequence
$\{(t_{k},x_{k},\hat{x}_{k})\}_{k\in\N}$ converges. Hence, by definition of
$\varphi_{k}$ and the fact that $\Theta_{k}>0$, we conclude that
\[
  0\leq \frac{k}{2}|x_{k}-\hat{x}_{k}|^{2}
    <U^{\rho}(t_{k},x_{k})-W(t_{k},\hat{x}_{k})
	\leq\sup_{(t,x,\hat{x})\in A}\big\{U^{\rho}(t,x)-W(t,\hat{x})\big\}<\infty,
\]
from which we find that
\[
  (\bar{t},\bar{x}):=\lim_{k\to\infty}(t_{k},x_{k})
	=\lim_{k\to\infty}(t_{k},\hat{x}_{k}).
\]
Combining $\Theta_{k}\geq\Theta$ and upper semicontinuity of $U^{\rho}-W$, we
furthermore see that
\begin{multline}\label{eq:proof-CP-limsup-to-0}
  0 \leq \limsup_{k\to\infty}\frac{k}{2}|x_{k}-\hat{x}_{k}|^{2}
    =\limsup_{k\to\infty}\{U^{\rho}(t_{k},x_{k})-W(t_{k},\hat{x}_{k})
      -\Theta_{k}\}\\
    \leq U^{\rho}(\bar{t},\bar{x})-W(\bar{t},\bar{x})-\Theta \leq 0.
\end{multline}
Finally, semicontinuity of $U$ and $W$ yields
\[
  \lim_{k\to\infty} U^{\rho}(t_{k},x_{k}) = U^{\rho}(\bar{t},\bar{x})
  \qquad\text{and}\qquad
  \lim_{k\to\infty}W(t_{k},\hat{x}_{k}) = W(\bar{t},\bar{x})
\]
as well as
\[
  \lim_{k\to\infty}\Theta_{k} = \Theta
    =U^{\rho}(\bar{t},\bar{x})-W(\bar{t},\bar{x}).
\]
From the last convergence we also obtain that $\bar{t}<T$, as else, using the
terminal inequality $U(T,\argdot)\leq W(T,\argdot)$,
\[
  0<\Theta
    =U^{\rho}(\bar{t},\bar{x})-W(\bar{t},\bar{x})
    =U(T,\bar{x})-W(T,\bar{x})+\frac{1}{\rho}\psi(T,\bar{x})\leq0,
\]
a contradiction. Hence, without loss of generality, we may assume $t_{k}<T$,
that is $(t_k,x_k),(t_k,\hat{x}_k)\in\intS$ for all $k\in\N$. From Ishii's
lemma, see Theorem 8.3 in \cite{crandall1992user}, for each $k$ we obtain
existence of matrices
$M_{k},\hat{M}_{k}\in\cS_{\tilde{m}}$ satisfying%
\footnote{Here, ${\bf I}_{\tilde{m}}$ denotes the identity matrix in
$\cS_{\tilde{m}}$.}
\begin{equation}\label{eq:Ishii-Matrix}
  \begin{pmatrix}
    M_{k} & 0 \\
    0 & -\hat{M}_{k}
  \end{pmatrix}
  \leq \begin{pmatrix}
    {\bf I}_{\tilde{m}} & -{\bf I}_{\tilde{m}} \\
    -{\bf I}_{\tilde{m}} & {\bf I}_{\tilde{m}}
  \end{pmatrix}
\end{equation}
and constants $q_k=-\hat{q}_k$ such that%
\footnote{Here, $\bar{J}^{2,+}U^{\rho}(t_{k},x_{k})$ and
$\bar{J}^{2,-}W(t_{k},\hat{x}_{k})$ denote the closures of second order
super- and subjets of $U^\rho$ and $W$, respectively.}
\begin{align*}
  \big(q_k,k(x_{k}-\hat{x}_{k}),M_{k}\big)
    \in\bar{J}^{2,+}U^{\rho}(t_{k},x_{k}),
  \quad
  \big(\hat{q}_k,k(x_{k}-\hat{x}_{k}),\hat{M}_{k}\big)
    \in \bar{J}^{2,-}W(t_{k},\hat{x}_{k}).
\end{align*}
By the perturbation result \Cref{lem:perturb} we know that $U^{\rho}$ is a
$\kappa/\rho$--strict viscosity subsolution of \eqref{eq:HJB}. Therefore, with
$\bar{\kappa}:=\inf_{(t,x,\hat{x})\in A}\{\kappa(t,x)\}>0$, we have
\[
  -\bar{\kappa}\geq\rF\bigl(t_{k},x_{k},q_k,k(x_{k}-\hat{x}_{k}),M_{k}\bigr).
\]
Combined with the supersolution property of $W$ and \cref{eq:Ishii-Matrix},
and Lipschitz-continuity of $f$ and $\Sigma$, this yields the existence of a
constant $C>0$ such that
\begin{align*}
  \bar{\kappa}
  &\leq\limsup_{k\to\infty}\Bigl[
    \rF\bigl(t_{k},x_{k},q_k,k(x_{k}-\hat{x}_{k}),M_{k}\bigr)
    -\rF\bigl(t_{k},\hat{x}_{k},\hat{q}_k,k(x_{k}-\hat{x}_{k}),\hat{M}_{k}\bigr)
  \Bigr]\\
  &\leq\limsup_{k\to\infty} \Bigl[
    Ck|x_{k}-\hat{x}_{k}|^{2}
    +\sup_{u\in U} \Bigl\{\tilde{k}(t_{k},x_{k},u)
    -\tilde{k}(t_{k},\hat{x}_{k},u)\Bigr\}
  \Bigr]\leq 0,
\end{align*}
where the last inequality is due to \cref{eq:proof-CP-limsup-to-0},
continuity of $\tilde{k}$, and compactness of $\cU$. Since $\kappa>0$, this
is the desired contradiction which concludes the proof.
\end{proof}

\subsection{Proof of \Cref{prop:dpp}}\label{subsec:dpp}

\begin{proof}[Proof of \Cref{prop:dpp}]
In what follows, we fix $(t,x)\in\S^n$ with $t<T$. We establish both
inequalities in the DPP separately.

\emph{Step 1}: Let $u\in\cU$ and choose a piecewise constant control
$\hat{u}\in\cA^n$ such that $\hat{u}=u$ on $[0,t+\delta_n]$. According to the
pseudo-Markov property established in Lemma 3.2.14 in
\cite{krylov2008controlled}, we have
\[
  \E\Bigl[\int_{t+\delta_n}^T\tilde{k}(s,\hat{X}^{\hat{u};t,x}_s,\hat{u}_s)\rd s
    + \tilde{g}(\hat{X}^{\hat{u};t,x}_s)\Big|\cF^W_{t+\delta_n}\Bigr]
  = \hat{\cJ}\bigl(\hat{u};t+\delta_n,\hat{X}^{u;t,x}_{t+\delta_n}\bigr).
\]
But then the tower property of conditional expectation yields
\begin{align*}
  \hat{V}(t,x)&=\E\Bigl[\int_t^{t+\delta_n}\tilde{k}(s,\hat{X}^{u;t,x}_s,u)\rd s
    +\hat{\cJ}\bigl(\hat{u};t+\delta_n,\hat{X}^{u;t,x}_{t+\delta_n}\bigr)\Bigr]\\
  &\geq \inf_{u\in\cU}\E\Bigl[\int_t^{t+\delta_n}
    \tilde{k}(s,\hat{X}^{u;t,x}_s,u)\rd s
    + \hat{V}^n\bigl(t+\delta_n,\hat{X}^{u;t,x}_{t+\delta_n}\bigr)\Bigr].
\end{align*}

\emph{Step 2}: We fix $\varepsilon>0$. Since
$\hat{\cJ}(u;t+\delta_n,\argdot)$ and $\hat{V}(t+\delta_n,\argdot)$ are
continuous functions (uniformly in $u\in\cA$) according to Lemma 3.2.2
in \cite{krylov2008controlled}, for any $y\in\R^{\tilde{m}}$ there exists
$\rho_y>0$ such that
\[
  |\hat{\cJ}(u;t+\delta_n,y) - \hat{\cJ}(u;t+\delta_n,\hat{y})|
    + |\hat{V}(t+\delta_n,y) - \hat{V}(t+\delta_n,\hat{y})|
    \leq \frac{1}{3}\varepsilon
\]
for all $u\in\cA$ and $\hat{y}\in B_{\rho_y}(y)$, where $B_{\rho_y}(y)$
is the open ball of radius $\rho_y$ centered around $y$. Now choose a sequence
$\{y_k\}_{k\in\N}$ in $\R^{\tilde{m}}$ such that
$\R^{\tilde{m}}=\bigcup_{k\in\N} B_{\rho_k}(y_k)$ and
$y_k\not\in B_{\rho_\ell}(y_\ell)$ whenever $k,\ell\in\N$ with $k\neq\ell$,
where we use the slightly abusive short-hand notation $\rho_k:=\rho_{y_k}$
for all $k\in\N$. From this, it follows that there exists a partition
$\{B_k\}_{k\in\N}$ of $\R^{\tilde{m}}$ of Borel sets such that $y_k\in B_k\subset
B_{\rho_k}(y_k)$ for all $k\in\N$. Next, for $k\in\N$, choose an
$(\varepsilon/3)$-optimal control $u_k\in\cA^n$ for $\hat{V}(t+\delta_n,y_k)$,
that is
\[
  \hat{\cJ}(u_k;t+\delta_n,y_k) \leq \hat{V}(t+\delta_n,y_k)
    + \frac{1}{3}\varepsilon,
\]
and it follows that
\[
  \hat{\cJ}(u_k;t+\delta_n,y)
  \leq \hat{\cJ}(u_k;t+\delta_n,y_k) + \frac{1}{3}\varepsilon
  \leq \hat{V}(t+\delta_n,y_k) + \frac{2}{3}\varepsilon
  \leq \hat{V}(t+\delta_n,y) + \varepsilon
\]
for all $y\in B_k$ and $k\in\N$. Now fix $u\in\cU$ and consider the control
$\hat{u}$ given by
\[
  \hat{u}:=u\quad\text{on }[0,t+\delta_n]
  \qquad\text{and}\qquad
  \hat{u}:=\sum_{k=1}^\infty
    u_k\mathds{1}_{\{\hat{X}^{u;t,x}_{t+\delta_n}\in B_k\}}
  \quad\text{on }(t+\delta_n,T].
\]
Clearly, $\hat{u}\in\cA^n$ and we conclude that
\begin{align*}
  \hat{V}^n(t,x)
    &\leq \hat{\cJ}(\hat{u};t,x)\\
    &= \E\Bigl[\int_t^{t+\delta_n}\tilde{k}(s,\hat{X}^{u;t,x}_s,u)\rd s
      + \sum_{k=1}^\infty \mathds{1}_{\{\hat{X}^{u;t,x}_{t+\delta_n}\in B_k\}}
      \hat{\cJ}\bigl(u_k;t+\delta_n,\hat{X}^{u;t,x}_{t+\delta_n}\bigr)\Bigr]\\
    &\leq \E\Bigl[\int_t^{t+\delta_n}\tilde{k}(s,\hat{X}^{u;t,x}_s,u)\rd s
      + \hat{V}^n\bigl(t+\delta_n,\hat{X}^{u;t,x}_{t+\delta_n}\bigr)\Bigr]
      + \varepsilon.
\end{align*}
Sending $\varepsilon\downarrow 0$ and taking the infimum over all $u\in\cU$
yields the result.
\end{proof}

\subsection{Proof of \Cref{th:supremum-is-supersolution}}
\label{app:supremum-stochastic-subsolutions}

The proof of \Cref{th:supremum-is-supersolution} is very similar to Theorem 4.1
in \cite{bayraktar2013stochastic}, and we follow the line of arguments very
closely.

\begin{proof}[Proof of \Cref{th:supremum-is-supersolution}]
\emph{Step 1}: The viscosity supersolution property in $\intS$.
Towards a contradiction, let $\varphi\in\cC^{1,2}(\intS)$ be a test function such
that $V^{-}-\varphi$ attains a strict global minimum equal to zero at some
$(\bar{t},\bar{x})\in\intS$ at which the viscosity supersolution property fails,
that is
\[
  \rF\bigl(\bar{x},\partial_{t}\varphi(\bar{t},\bar{x}),
    \rD_{x}\varphi(\bar{t},\bar{x}),\rD_{xx}\varphi(\bar{t},\bar{x})\bigr)<0.
\]
By continuity of $F$, we can find $\varepsilon>0$ such that
$\bar{t}+\varepsilon<T$ and such that
\begin{equation}\label{eq:proof-SSub-contr-in-ball}
  \rF\bigl(t,x,\partial_{t}\varphi(t,x),
    \rD_{x}\varphi(t,x),\rD_{xx}\varphi(t,x)\bigr)<0
    \quad \forall (t,x)\in B_{\varepsilon}(\bar{t},\bar{x}),
\end{equation}
where $B_{\epsilon}(\bar{t},\bar{x})$ denotes the open ball of radius
$\varepsilon$
around $(\bar{t},\bar{x})$ taken relative to $\S$. We also write
$\overline{B_{\epsilon}(\bar{t},\bar{x})}$ for the closure of
$B_{\epsilon}(\bar{t},\bar{x})$. Since $V^{-}-\varphi\in LSC(\S)$, the set
$\overline{B_{\varepsilon}(\bar{t},\bar{x})}\setminus
B_{\varepsilon/2}(\bar{t},\bar{x})$ is compact, and the minimum of $V^{-}-\varphi$
is strict, we can find a constant $\kappa>0$ such that
\begin{equation}\label{eq:proof-SSub-kappa}
  V^{-}-\kappa\geq\varphi \quad\text{on }\quad
    \overline{B_{\varepsilon}(\bar{t},\bar{x})}\setminus
    B_{\varepsilon/2}(\bar{t},\bar{x}).
\end{equation}
With this, we define for $\eta\in(0,\kappa)$ the function
$\varphi^{\eta}:=\varphi+\eta$, for which it holds that
\[
  \varphi^{\eta}(\bar{t},\bar{x})
    =\varphi(\bar{t},\bar{x})+\eta
    =V^{-}(\bar{t},\bar{x})+\eta
    >V^{-}(\bar{t},\bar{x}).
\]
Moreover, we define another function $W^\eta:\S\to\R$ by
\begin{align*}
  W^{\eta}:=\begin{cases}
    V^{-}\vee \varphi^{\eta}
      &\text{on }\overline{B_{\varepsilon}(\bar{t},\bar{x})},\\
    V^{-} &\text{otherwise}.
\end{cases}
\end{align*}
Then, from \cref{eq:proof-SSub-kappa} and using $\kappa>\eta$, it follows that 
\[
  V^{-}>\varphi^{\eta}\qquad\text{on }
    \overline{B_{\varepsilon}(\bar{t},\bar{x})}\setminus
    B_{\varepsilon/2}(\bar{t},\bar{x}),
\]
so $W^{\eta}=V^{-}$ outside of $B_{\varepsilon/2}(\bar{t},\bar{x})$. In
particular, $W^{\eta}$ is lower semicontinuous. Next, note that
\begin{equation}\label{eq:proof-SSub-contr.}
  W^{\eta}(\bar{t},\bar{x}) = \varphi^{\eta}(\bar{t},\bar{x})
    > V^{-}(\bar{t},\bar{x})
\end{equation}
and, by choice of $\varepsilon$, we have $W^{\eta}(T,\argdot)=V^{-}(T,\argdot)$.
As $W^\eta = V^{-}$ outside a bounded set, $W^\eta$ satisfies the same growth
condition as $V^-$. Therefore, in order to show that $W^\eta$ is a stochastic
subsolution, we are left with verifying that it satisfies the submartingale
property. Once this is achieved, \cref{eq:proof-SSub-contr.} is a contradiction
to the maximality of $V^{-}$ and hence the viscosity supersolution property is
established.
Towards the verification of the submartingale property, let us fix $(t,x)\in\S$
and denote by $U^{t,x}\in\cA^{weak}(t,x)$ any weak admissible control.
We agree that whenever we subsequently speak of a submartingale, we always mean
with respect to $(\F^{t,x},\P^{t,x})$.
Next, we take as given two $\F^{t,x}$-stopping times
$t\leq\tau\leq\rho\leq T$ and we define the event
\[
  A:=\bigl\{(\tau,X_{\tau}^{t,x})\in B_{\varepsilon/2}(\bar{t},\bar{x})
  \text{ and }
  \varphi^\eta(\tau,X_{\tau}^{t,x})>V^{-}(\tau,X_{\tau}^{t,x})\bigr\}
\]
as well as the stopping time
\[
  \tau_1:=\inf\bigl\{s\in[\tau,T]:X_{s}^{t,x}\in\partial
    B_{\varepsilon/2}(\bar{t},\bar{x})\bigr\}.
\]
Here, $\partial B_{\varepsilon/2}(\bar{t},\bar{x})$ denotes the boundary of
$B_{\varepsilon/2}(\bar{t},\bar{x})$.
Note that it follows from \cref{eq:proof-SSub-contr-in-ball} that we have
\begin{equation}\label{eq:proof-SSub-contr-in-ball-eta}
  \rF\bigl(t,x,\partial_t \varphi^\eta(t,x),\rD_x\varphi^\eta(t,x),
    \rD_{xx}\varphi^{\eta}(t,x)\bigr)<0
    \quad\forall (t,x)\in B_{\varepsilon}(t,x)
\end{equation}
as the derivatives of $\varphi$ and $\varphi^{\eta}$ agree. With this, using
It\^{o}'s formula and \cref{eq:proof-SSub-contr-in-ball-eta}, we find that
\begin{align*}
  \mathds{1}_{A}W^{\eta}(\tau,X_{\tau}^{t,x})
    &=\mathds{1}_{A}\varphi^{\eta}(\tau,X_{\tau}^{t,x})\\
    &\leq \E^{t,x}\Bigl[\Bigl(\int_{\tau}^{\rho\wedge\tau_{1}}
      \tilde{k}(s,X_{s}^{t,x},u_{s})\rd s 
      +\varphi^{\eta}\bigl(\rho\wedge\tau_{1},
      X_{\rho\wedge\tau_{1}}^{t,x}\bigr)\Bigr)\mathds{1}_{A}
      \Big|\cF^{t,x}_{\tau}\Bigr]\\
	&\leq \E^{t,x}\Bigl[\Bigl(\int_{\tau}^{\rho\wedge\tau_{1}}
      \tilde{k}(s,X_{s}^{t,x},u_{s})\rd s
      + W^{\eta}\bigl(\rho\wedge\tau_{1},
      X_{\rho\wedge\tau_{1}}^{t,x}\bigr)\Bigr)\mathds{1}_{A} 
      \Big|\cF^{t,x}_{\tau}\Bigr],
\end{align*}
where the stochastic integral vanishes by the boundedness of the integrand.
Next, using that $V^-$ is a stochastic supersolution and hence satisfies
the submartingale property, we obtain
\begin{align*}
  \mathds{1}_{A^{c}}W^{\eta}(\tau,X_{\tau}^{t,x})
    &=\mathds{1}_{A^{c}}V^{-}(\tau,X_{\tau}^{t,x})\\
    &\leq \E^{t,x}\Bigl[\Bigl(\int_{\tau}^{\rho\wedge\tau_{1}}
      \tilde{k}(s,X_{s}^{t,x},u_{s})\rd s 
      +V^{-}\bigl(\rho\wedge\tau_{1},X_{\rho\wedge\tau_{1}}^{t,x}\bigr)
      \Bigr)\mathds{1}_{A^{c}}\Big|\cF^{t,x}_{\tau}\Bigr]\\
    &\leq \E^{t,x}\Bigl[\Bigl(\int_{\tau}^{\rho\wedge\tau_{1}}
      \tilde{k}(s,X_{s}^{t,x},u_{s})\rd s 
      +W^{\eta}\bigl(\rho\wedge\tau_{1},X_{\rho\wedge\tau_{1}}^{t,x}\bigr)
      \Bigr)\mathds{1}_{A^{c}}\Big|\cF^{t,x}_{\tau}\Bigr].
\end{align*}
Putting everything together yields
\[
  W^{\eta}(\tau,X_{\tau}^{t,x})
    \leq \E^{t,x}\Bigl[\int_{\tau}^{\rho\wedge\tau_{1}}
    \tilde{k}(s,X_{s}^{t,x},u_{s})\rd s 
    +W^{\eta}\bigl(\rho\wedge\tau_{1},X_{\rho\wedge\tau_{1}}^{t,x}\bigr)
    \Big|\cF^{t,x}_{\tau}\Bigr].
\]
We now define $B:=\{\rho>\tau_{1}\}\in\cF^{t,x}_{\tau_{1}\wedge\rho}$.
Using the submartingale property of $V^-$ again, we get
\begin{align*}
  \mathds{1}_{B}W^{\eta}(\tau_{1},X_{\tau_{1}}^{t,x})
    &=\mathds{1}_{B}V^{-}(\tau_{1},X_{\tau_{1}}^{t,x})\\
    &\leq \E^{t,x}\Bigl[\Bigl(\int_{\tau_{1}}^{\rho}
      \tilde{k}(s,X_{s}^{t,x},u_{s})\rd s
      +V^{-}\bigl(\rho,X_{\rho}^{t,x}\bigr)\Bigg)\mathds{1}_{B}
      \Big|\cF^{t,x}_{\tau_{1}}\Bigr]\\
    &\leq \E^{t,x}\Bigl[\Bigl(\int_{\tau_{1}}^{\rho}
      \tilde{k}(s,X_{s}^{t,x},u_{s})\rd s 
      +W^{\eta}\bigl(\rho,X_{\rho}^{t,x}\bigr)\Bigr)\mathds{1}_{B}
      \Big|\cF^{t,x}_{\tau_{1}}\Bigr].
\end{align*}
Using this inequality in our estimate above we arrive at
\[
  W^{\eta}(\tau,X_{\tau}^{t,x})
  \leq \E^{t,x}\Bigl[\int_{\tau}^{\rho}
    \tilde{k}(s,X_{s}^{t,x},u_{s})\rd s 
    +W^{\eta}\bigl(\rho,X_{\rho}^{t,x}\bigr)\Big|\cF^{t,x}_{\tau}\Bigr],
\]
which is the desired submartingale property.

\emph{Step 2}: The supersolution property at terminal time. As in the first
step, we argue by contradiction and assume that there exists $\bar{x}\in\S$
such that
\[
  V^{-}(T,\bar{x})-\tilde{g}(\bar{x})=:-\bar{\kappa}<0.
\]
As $g$ is continuous and $\mu$ compactly supported, we can find 
$\varepsilon\in(0,\bar{\kappa})$ such that
\begin{equation}\label{eq:proof-Stochsub-term-ineq-in-ball}
  V^{-}(T,x)-\tilde{g}(x)<-\varepsilon<0
  \quad\forall x\in B_{\varepsilon}(T,\bar{x}).
\end{equation}
Furthermore, since $V^{-}\in LSC(\S)$ and
$\overline{B_{\varepsilon}(T,\bar{x})}\setminus B_{\varepsilon/2}(T,\bar{x})$
is compact, we can find a constant $\beta>0$ such that
\[
  V^{-}(T,\bar{x})+\varepsilon
  <\frac{\varepsilon^{2}}{2\beta}+V^{-}(t,x)
  \quad \forall (t,x)\in \overline{B_{\varepsilon}(T,\bar{x})}\setminus
  B_{\varepsilon/2}(T,\bar{x}).
\]
We now show that there exists $M\geq \varepsilon/(2\beta)$ large enough such
that
\[
  \phi(t,x):=V^{-}(T,\bar{x})-\frac{1}{\beta}|x-\bar{x}|^{2}-M(T-t)
\]
satisfies
\[
  -\phi_{t}(t,x)-\inf_{u\in \cU}\cL^{u}\phi(t,x)<0
  \quad \forall (t,x)\in\overline{B_{\varepsilon}(T,\bar{x})},
\]
For this, we first note that
\[
  \phi_{t}(t,x) = M,\quad
  \rD_x\phi(t,x)=\frac{2}{\beta}(x-\bar{x}),
  \quad
  \rD_{xx}\phi(t,x)=\frac{2}{\beta}{\bf I}_{\tilde{m}},
  \quad (t,x)\in\overline{B_{\varepsilon}(T,\bar{x})}
\]
which together with boundedness of $f$ and $\Sigma$ implies the existence
of $C>0$ such that
\begin{align*}
  -\phi_{t}(t,x)-\inf_{u\in \cU}\cL^{u}\phi(t,x)
  \leq -M + C\varepsilon\frac{2}{\beta}+C\frac{1}{\beta},
  \qquad (t,x)\in\overline{B_{\varepsilon}(T,\bar{x})}
\end{align*}
and the right-hand side is negative for $M$ sufficiently large.
Next, we note that for any $(t,x)\in
\overline{B_{\varepsilon}(T,\bar{x})}\setminus B_{\varepsilon/2}(T,\bar{x})$
 we have
\[
  \phi(t,x) 
  \leq V^{+}(t,x)-\varepsilon - \frac{1}{\beta}|x-\bar{x}|^{2} - M(T-t)
  \leq V^{-}(t,x)-\varepsilon,
\]
where we used that $|x-\bar{x}|\geq \varepsilon/2$ and $T-t\geq \varepsilon/2$.
By \cref{eq:proof-Stochsub-term-ineq-in-ball}, it follows that
\[
  \phi(T,x)
  \leq V^{-}(T,\bar{x})
  \leq \tilde{g}(x)-\varepsilon,\quad (T,x)\in \overline{B_{\varepsilon}(T,\bar{x})}.
\]
Now set $\phi^{\eta}:=\phi+\eta$ and define $W^\eta:\S\to\R$ by
\[
  W^{\eta}:=\begin{cases}
    V^{-}\vee \phi^{\eta} &\text{on }\overline{B_{\varepsilon}(T,\bar{x})},\\
    V^{-} &\text{otherwise}.
	\end{cases}
\]
We may proceed exactly as in the first step of this proof to show that
$W^{\eta}\in\cV^{-}$ and obtain the same contradiction to the maximality of
$V^{-}$.
\end{proof}

\bibliographystyle{alpha}
\bibliography{References}

\end{document}